\documentclass[12pt,a4paper]{amsart}
\usepackage{amscd}
\usepackage{amsmath}
\usepackage{latexsym}
\usepackage{amsfonts}
\usepackage{amssymb}
\usepackage{amsthm}
\usepackage{graphicx}
\usepackage{verbatim}
\usepackage{mathrsfs}
\usepackage{enumerate}
\usepackage{color}

\newtheorem{definition}{Definition}[section]
\newtheorem{theorem}{Theorem}[section]
\newtheorem{lemma}[theorem]{Lemma}

\newtheorem{proposition}[theorem]{Proposition}

\theoremstyle{remark}
\newtheorem{remark}[theorem]{Remark}

\newcommand{\ZZ}{\mathbb{Z}}
\newcommand{\NN}{\mathbb{N}}

\newcommand{\RR}{\mathbb{R}}
\DeclareMathOperator{\Div}{div}
\allowdisplaybreaks
\begin{document}
\title[Minimal blow-up data for NS in critical Fourier-Herz space ]{Minimal blow-up initial data in critical Fourier-Herz \\spaces for potential Navier-Stokes singularities }

\author[J. Li, C. Miao and  X. Zheng]{Jingyue Li, Changxing Miao and Xiaoxin Zheng}

\address{The Graduate School of China Academy of Engineering Physics, Beijing 100088, P.R. China}
\email{m\_lijingyue@163.com}

\address{ Institute of Applied Physics and Computational Mathematics, P.O. Box 8009, Beijing 100088, P.R. China.}
\email{miaochangxing@iapcm.ac.cn}
\address{ School of Mathematics and Systems Science, Beihang University, Beijing 100191, P.R. China}
\email{xiaoxinzheng@buaa.edu.cn}

\date{\today}

\begin{abstract}
In this paper, we mainly prove the existence of the minimal blow-up initial data in critical Fourier-Herz space $F\dot{B}^{2-{\frac3p}}_{p,q}(\RR^3)$ with $1<p\leq\infty$ and $1\leq q<\infty$ for the three dimensional  incompressible potential Navier-Stokes equations by developing  techniques of ``localization in space" involving the partial regularity given by the De Giorgi iteration, weak-strong uniqueness, the short-time behaviour of the kinetic energy and stability of singularity of Calder\'on's solution.
\end{abstract}

\maketitle
\section{Introduction}
In this paper, we consider  Cauchy problem of the three dimensional incompressible Navier-Stokes equations
\begin{equation}\label{NS}\tag{NS}
\left\{\begin{array}{ll}
\partial_t u-\nu\Delta u+u\cdot\nabla u+\nabla P=0\\
\Div u=0\\
u(x,0)=u_0,
\end{array}\right.
\end{equation}
where the vector field $u$  and the scalar function $P$  describe the velocity field and the associated pressure of fluid, respectively. And $\nu>0$ is the kinematic viscosity. The initial velocity $u_0$ satisfies $\Div u_0=0.$
It is well-known that problem \eqref{NS} has the natural scaling, that is, if $u$ is a solution of system \eqref{NS} with initial data $u_0$, then so is $u_{\lambda}$, for any $\lambda>0$,  associated with initial data $u_{0\lambda}$,  where
\begin{equation}\label{scall}
u_{\lambda}(x,t)\triangleq\lambda u(\lambda x,\lambda^2 t),\quad
u_{0\lambda}\triangleq\lambda u_0(\lambda x).
\end{equation}
We call Banach space $X$ is  \textit{critical space} if it's norm $\|\cdot\|_{X}$ is invariant under scaling~\eqref{scall},  for example, the Lebesgue space $L^3(\RR^3)$, homogeneous Sobolev space $\dot{H}^{{1\over 2}}(\RR^3)$, homogeneous Besov space $\dot{B}^{-1+{3\over p}}_{p,q}(\RR^3)$, homogeneous Fourier-Herz space $F\dot{B}^{2-{3\over p}}_{p,q}(\RR^3)$  and so on.

The incompressible Navier-Stokes equations has been studied by many researchers in the past years. Especially, there are many interesting results concerning the mild solution $u$ which  solves the following integral equations:
\[u(x,t)=e^{\nu t\Delta}u_0+\int^t_0 e^{\nu(t-s)\Delta}\mathbb{P}(u\cdot\nabla u)\,\mathrm{d}s,\]
where $\mathbb{P}$ is the Leray projector. For the separable critical space, Fujita and Kato \cite{FK64} firstly established the local well-posedness for the large initial data in $\dot{H}^{1/2}(\RR^3)$, as well as the global well-posedness for small initial data in $\dot{H}^{{1/2}}(\RR^3)$. Later on, Kato \cite{Ka84} and Cannone \cite{Can97,Can04} generalized such well-posed theory to $L^3(\RR^3)$ and $\dot{B}^{-1+{3\over p}}_{p,q}(\RR^3)$ with $1\leq p,q<\infty$, respectively.  For the inseparable critical spaces, there only exists the global well-posed theory for the small initial data by now. Cannone  \cite{Can97,Can04} proved the global well-posedness to \eqref{NS} for small data in $\dot{B}^{-1+{3\over p}}_{p,\infty}(\RR^3)$ with $3<p<\infty$, and  Koch and Tataru \cite{KT01} showed the global well-posedness to system \eqref{NS} for small initial data in $\rm BMO^{-1}(\RR^3)$. Since Besov space $\dot{B}^{-1+\frac3p}_{p,\infty}$ with $p\in(3,\infty)$ and $\rm BMO^{-1}(\RR^3)$  contain non-trivial homogenous functions of degree $-1$, the both results mean that the  global existence and uniqueness of the small self-similar solution. But the local well-posedness for large initial data in the inseparable critical space remains an open problem.

Let $X$ be a separable critical space. Denoted by  $NS(u_0)$ the local in time mild solution to \eqref{NS} starting from $u_0\in X$, ones  consider the following question based on the above local well-posed theory:

\textbf{Question:}

\emph{Suppose there exist a initial data $u_0$ in the separable critical space $X$ such that the maximal existence time of $NS(u_0)$ is finite. Then does there exist a minimal blow-up initial data $v_0\in X$, that is, does there exist $v_0\in X$ with minimal norm such that the maximal existence time of $NS(v_0)$ is finite?}

Rusin and \v{S}ver\'{a}k \cite{RS11} firstly considered this question in the space $\dot{H}^{{1\over 2}}(\RR^3)$ and gave a positive answer. The main ingredients of their proof consist of the singularities of local mild solution inside the ball at the maximal existence time, the stability of singularities and the compactness of  $\dot{H}^{{1/2}}(\RR^3)\hookrightarrow L^2_{\rm loc}(\RR^3)$. Later, Gallagher, Koch and Planchon in \cite{GKP13,GKP16} utilized  concentration compactness argument and profile decomposition to give a unified affirmative answer to this question in $\dot{H}^{{1/2}}(\RR^3)$, $L^3(\RR^3)$ and $\dot{B}^{-1+{3\over p}}_{p,q}(\RR^3)$ with $3<p,q<\infty$. Especialy for $L^3(\RR^3)$, Jia and \v{S}ver\'{a}k in \cite{JS13} gave a simple proof of the existence of minimal blow-up initial data   by exploiting the regularity of energy solution in short time which compensates the lack of compactness of the embedding $L^3(\RR^3)\hookrightarrow L^2_{\rm loc}(\RR^3)$.

Recently, some well-posed results of problem \eqref{NS} for large initial data in Fourier-Herz spaces were established. For example, Cannone and Karch \cite{CK04} considered existence and uniqueness of singular solution in the space of pseudomeasure $F\dot{B}^{-1}_{\infty,\infty}(\RR^3)$. Lei and Lin \cite{LL11} proved an interesting global well-posed result to problem \eqref{NS} if the initial data belongs to the space
$$\chi^{-1} \triangleq\left \{ u\in \mathcal{D}'(\RR^3)\,\Big|\,\,\int_{\RR^3} |\xi|^{-1} |\hat{u}|(\xi)\,\mathrm{d}\xi<\infty\right\}$$ which coincides with  $F\dot{B}^{-1}_{1,1}(\RR^3)$, and the corresponding norm is bounded exactly by the viscosity coefficient $\nu$. Subsequently, Cannone and Wu \cite{CW12} extended this result to larger space $F\dot{B}^{-1}_{1,q}(\RR^3)$ for all $1\leq q\leq 2$. Recently, Li and Zheng \cite{LZ17} obtained the local well-posed result in critical Fourier Herz spaces $F\dot{B}^{2-{3\over p}}_{p,q}(\RR^3)$ with $q\neq \infty$ for the large initial data, and the global well-posed theory  in  $F\dot{B}^{2-{3\over p}}_{p,q}(\RR^3)$  for small initial data. A nature question arises: does there exist a blow-up initial data in the framework of the critical Fourier-Herz space. To solve this question,  we face some difficulties which don't appear in known results. For example,  the embedding that $F\dot{B}^{2-\frac3p}_{p,q}(\RR^3)\hookrightarrow L^2_{\rm loc}(\RR^3)$
doesn't hold as long as $p<\frac32.$ This leads to the method used in \cite{JS13} doesn't work directly.
In order to overcome this difficulty, we adopt the Calder\'on splitting argument to decompose the local mild solution $u(x,t)$ on $[0,T^*)$ into two parts
\[u=v+w.\]
For this purpose, we will establish abstract interpolation theory of Fourier-Herz spaces $F\dot{B}^s_{p,q}(\RR^d)$ by using $K$-function, see Lemma \ref{lem.inter}. This allows us to decompose $u_0$ into two parts as follows:
\[u_0=v_0+w_0,\]
where  $v_0$ and $w_0$ fulfill $\Div v_0=\Div w_0=0$ and
\begin{equation}\label{decom.1}
\|v_0\|_{F\dot{B}^{2-\frac{3}{\tilde{p}}+\varepsilon}_{\tilde{p},\tilde{q}}}\leq C2^{-j\theta}\alpha\quad\text{and}\quad
\|w_0\|_{L^2}\leq C2^{j(1-\theta)}\alpha\quad\, \text{for each}\,\,\,\varepsilon>0.
\end{equation}
Since $v_0\in F\dot{B}^{2-\frac{3}{\tilde{p}}+\varepsilon}_{\tilde{p},\tilde{q}}(\RR^3)$, by the local well-posed theory in the subcritical space, we know that the following system \begin{equation}\label{V-introd}
\left\{\begin{array}{ll}
\partial_t v-\nu\Delta v+v\cdot\nabla v+\nabla Q=0,\\
\Div v=0,\\
v(x,0)=v_0,
\end{array}\right.
\end{equation}
exists a unique local mild solution $v(x,t)$ satisfying $$v(x,t)\in C_{\rm b }\Big([0,T];F\dot{B}^{2-\frac{3}{\tilde{p}}+\varepsilon}_{\tilde{p},\tilde{q}}(\RR^3)\Big)\cap \widetilde{L}^r\Big([0,T];\,F\dot{B}^{2-\frac{3}{\tilde{p}}+\frac2r+\varepsilon}_{\tilde{p},\tilde{q}}(\RR^3)\Big)$$
where $T>T^*$.  Letting the difference $w(x,t)=u(x,t)-v(x,t)$,  ones  verify  by the fact $w_0\in L^2(\RR^3)$ that $w(x,t)$ is a Leray solution of the following perturbed problem
\begin{equation}\label{W-1}
\left\{\begin{array}{ll}
\partial_t w-\nu\Delta w+w\cdot\nabla w+w\cdot\nabla v+v\cdot\nabla w+\nabla \bar{P}=0\\
\Div w=0\\
w(x,0)=w_0\in L^2(\RR^3).
\end{array}\right.
\end{equation}
Since $v(x,t)$ is regular at $T=T^*$,   the singularity of $u$  at the maximal time $T^*$  is caused by $w$. This requires us to study the singularity of $w$ at  $T^*$.
To do this,  we first prove the $\varepsilon$-regularity criterion of the suitable weak solution by De Giorgi iteration and dimensional analysis. Next, making good use of the splitting argument, trilinear estimates and the regularity structure of the initial data, we  obtain the short-time behaviour of the energy solution, and then we  get the weak-strong uniqueness of solutions by  smoothing effect of the  heat kernel and uniform $L^2$-estimate.  These properties enable us to conclude that the singularities of $w(x,t)$ only occur inside the ball. Based on this, we further get by the  compactness argument that  the stability of singularity. With these properties, we eventually obtain that there exists a blow-up initial data to problem \eqref{NS} in the framework of Fourier-Herz space.

 We give some notations before the presentation of the main result. Set
\[\mathcal{B}_{\rho}=\left\{u_0\in F\dot{B}^{2-{3\over p}}_{p,q}(\RR^3)\,\Big |\,\Div u_0=0,\,\,\|u_0\|_{F\dot{B}^{2-{3\over p}}_{p,q}}<\rho\right\},\]
and
\[\rho_{ \max}=\sup\left\{\rho>0\,\big|\,u_0\in\mathcal{B}_{\rho},\,\,T^*(u_0)=\infty\right\}\]
where $T^*(u_0)$ is the maximal existence time of $NS(u_0)$.
Now, the main theorem reads:
\begin{theorem}\label{Thm}
Let $1<p\leq \infty$ and $1\leq q<\infty$. Suppose $\rho_{\max}<\infty$. Then, there exist a divergence-free initial data $u_0\in F\dot{B}^{2-{3\over p}}_{p,q}(\RR^3)$ with $\|u_0\|_{F\dot{B}^{2-{3\over p}}_{p,q}(\RR^3)}=\rho_{ \max}$ such that
\[T^*(u_0)<\infty.\]
Moreover, $\mathcal{M}$ is compact in $F\dot{B}^{2-{3\over p}}_{p,q}(\RR^3)$ modulo translations and scaling \eqref{scall}, where
\[\mathcal{M}=\left\{u_0\in F\dot{B}^{2-{3\over p}}_{p,q}(\RR^3)\,\Big|\, \Div u_0=0,\,\,\, \|u_0\|_{F\dot{B}^{2-{3\over p}}_{p,q}(\RR^3)}=\rho_{\max},\, \,\,T^*(u_0)<\infty\right\}.\]
\end{theorem}
\begin{remark}\label{rem.diff}
By the Hausdorff-Young inequality and Plancherel theorem, we easily have that
\begin{itemize}
	\item  for $1<p<\frac32,$
	$F\dot{B}^{2-\frac3p}_{p,q}(\RR^3)\hookrightarrow\dot{B}^{-1+\frac3{p'}}_{p',q}(\RR^3),\quad \frac 1p+\frac 1{p'}=1$;\medskip
	
	\item  for $p=\frac32,$
	$F\dot{B}^{0}_{\frac32,\frac32}(\RR^3)\hookrightarrow L^3(\RR^3);$
	\medskip
	
	\item for $p=2,$ $F\dot{B}^{\frac12}_{2,2}(\RR^3)\sim \dot{H}^{\frac12}(\RR^3).$
\end{itemize}
This implies that  Theorem \ref{Thm} includes the result was shown in \cite{RS11},  and builds the relationship between  Theorem \ref{Thm} and results in   \cite{GKP13,GKP16,JS13} in some sense. But, our argument and
technique are different with that of papers  \cite{GKP13,GKP16}, whose proof strongly relies on  ``profile decomposition".   More importantly, we develop some useful techniques of  ``localization in space" including the partial regularity, weak-strong uniqueness, the short-time behaviour of the kinetic energy and stability of singularity of weak solution to the perturbed problem~\eqref{W-1}, which will be powerful  in the study of the incompressible fluid equations, such as MHD system.
\end{remark}
The rest of this paper is structured as follows. In Section \ref{sec.2}, we review the definition and some properties of Fourier-Herz spaces, and give several useful lemmas including abstract interpolation theory of Fourier-Herz spaces based on $K$-function.  In Section \ref{sec.3}, we establish the well-posedness theory of mild solution to problem \eqref{NS} in Fourier-Herz spaces.  Section~\ref{sec.4}  is devoted to developing techniques of ``localization in space" to  the perturbed problem \eqref{W-1}.
In Section \ref{sec.5}, we give  the proof of Theorem \ref{Thm} by using the properties established  in foregoing sections.

\subsection*{Notation}
We denote $C$ as an absolute positive constant. $C(\lambda, \beta,\cdots)$ denotes a positive constant depending only on $\lambda, \beta, \cdots$. We adopt the convention that nonessential constants $C$ may change from line to line and the usual Einstein summation convention. Given two quantities $a$ and $b$, we denote $a\lesssim b$ and $a\lesssim_{\lambda,\beta,\ldots} b$ as $a\leq C b$ and $a\leq C(\lambda,\beta,\cdots)\, b$ respectively. For any $x_0\in\RR^3$ and $t_0\in \RR^+$, $B_R(x_0)\in \RR^3$ means a ball with radius $R$ centered at $x_0$ and $Q_R(x_0,t_0)=B_R(x_0)\times (t_0-R^2,t_0)\in \RR^3\times \RR$. 

\section{Preliminaries}\label{sec.2}
\setcounter{section}{2}\setcounter{equation}{0}

In this section, we first recall some facts concerning  Littlewood-Paley theory and some useful lemmas which will be used in subsequent sections, see for example \cite{Can04,MWZ12}. Next, we
establish abstract interpolation theory of Fourier-Herz spaces in terms of  $K$-function, which allows us  to perform the Calder\'on argument.

Let $\mathcal{S}(\RR^d)$ be the Schwartz class of rapidly decreasing functions and $\mathcal{S}'(\RR^d)$ be its dual space. We denote by $\hat{f}$ the Fourier transform of $f$. For each $1\leq p\leq \infty$, ones define
\[FL^p(\RR^d)\triangleq\left\{f\in \mathcal{S}'(\RR^d)\,\Big|\,\|f\|_{FL^p(\RR^d)}=\|\hat{f}\|_{L^p(\RR^d)}<\infty\right\}.\]
Let $\mathcal{C}$ be the annulus $\{\xi\in\RR^d\,|\,{3\over 4}\leq |\xi|\leq {8\over 3}\}$. Let $\phi,\,\chi\in \mathcal{S}(\RR^d)$ satisfying $\hat{\phi}\in \mathcal{D}(\mathcal{C})$ with $0\leq \hat{\phi}\leq 1$ and $\hat{\chi}\in \mathcal{D}(B_{4\over 3}(0))$ with $0\leq \hat{\chi}\leq 1$. For any $j\in \ZZ$, define $\phi_j(x)=2^{jd}\phi(2^jx)$ and $\chi_j(x)=2^{jd}\chi(2^jx)$. Then  we define
\[\dot{\Delta}_j=\phi_j\ast\cdot\quad\text{ and }\quad \dot{S}_{j}=\chi_j\ast\cdot.\]
From the definitions of the frequency localization operators,  we easily find that
 $\forall f\in \mathcal{S}'(\RR^d)$,
\begin{equation}\label{orth}
\dot{\Delta}_j\dot{\Delta}_k f=0  \qquad\text{ for}\quad  |j-k|>1.
\end{equation}
This implies
\[\tilde{\dot{\Delta}}_j\dot{\Delta}_j=\dot{\Delta}_j,\quad
\tilde{\dot{\Delta}}_j=\dot{\Delta}_{j-1}+\dot{\Delta}_j+\dot{\Delta}_{j+1}.\]
Now we give the definition of Fourier-Herz space.
\begin{definition} \label{def.fh}

Let $f\in \mathcal{S}'_h(\RR^d)=\mathcal{S}(\RR^d)/\mathcal{P}$, where $\mathcal{P}$ is the set of all polynomials, then we say $f\in F\dot{B}^{s}_{p,q}(\RR^d)$ with $s\in\RR$ and $(p,q)\in[1,\infty]^2$, if $\|f\|_{F\dot{B}^{s}_{p,q}}<\infty$, where
\begin{equation*}
\|f\|_{F\dot{B}^{s}_{p,q}} \triangleq\begin{cases}
\displaystyle \bigg(\sum_{j\in\ZZ}2^{jsq}\|\dot{\Delta}_jf\|^q_{FL^p}\bigg)^{1\over q}\quad \quad &q<\infty\\
\displaystyle \sup_{j\in\ZZ}2^{js}\|\dot{\Delta}_jf\|_{FL^p}\quad\quad &q=\infty
\end{cases}.
\end{equation*}
When $s<d(1-\frac1p)$, $F\dot{B}^{s}_{p,q}(\RR^d)$ is a Banach space.
\end{definition}

 Let us remark that when $s<0,$  Fourier-Herz space  $F\dot{B}^s_{p,q}(\RR^d)$  can be characterized  by Gaussian kernel $e^{t\Delta}$.  Specifically:
\begin{lemma}[\cite{BZ18}]\label{norm}
Let $s<0$ and $1\leq p,q\leq\infty$. Then for any $f\in F\dot{B}^{s}_{p,q}(\RR^d)$, there exits constants $C_1>0$ and $C_2>0$ such that
\[C_1\|f\|_{F\dot{B}^s_{p,q}}\leq \big\|\|t^{-{s\over 2}}\widehat{e ^{t\Delta}f}\|_p\big\|_{L^q(\RR^+;\,{\mathrm{d}t\over t})}\leq C_2\|f\|_{F\dot{B}^s_{p,q}}.\]
\end{lemma}
Taking into account the time variable, we give the definition of the mixed time-space Fourier-Herz space which is called the so called ``Chemin-Lerner" space.
\begin{definition}
Let $s\in\RR$ and $(p,q,r)\in [1,\infty]^3$. We say $$u\in \widetilde{L}^r([\alpha,\beta];\,F\dot{B}^s_{p,q})$$ if and only if
$$\|u\|_{\widetilde{L}^r([\alpha,\beta];\,F\dot{B}^s_{p,q})}\triangleq \big\|{2^{js}\|\dot{\Delta}_ju\|_{L^r([\alpha,\beta];\,FL^p)}}\big\|_{\ell^q}<\infty.$$
\end{definition}
In the setting of $\widetilde{L}^r([\alpha,\beta];\,F\dot{B}^s_{p,q})$,  we can get by the Young inequality that
\begin{equation}\label{est.TS}
\|e^{t\Delta}u_0\|_{\widetilde{L}^r([0,T);\,F\dot{B}^s_{p,q})}\leq C\|u_0\|_{F\dot{B}^s_{p,q}}\qquad \text{for all}\quad s\in \RR, 1\leq r,p,q\leq \infty.
\end{equation}
Next, we establish an useful abstract interpolation theory of Fourier-Herz spaces $F\dot{B}^s_{p,q}(\RR^d)$, which is a key point in our proof of Theorem \ref{Thm}. To do this, we need to introduce the generalized Fourier-Herz spaces  $F\dot{B}^s_{p,q,(r)}(\RR^d)$.
\begin{definition}\label{def.fh1}
Let $ s\in\RR$, $1< p<\infty$, $1\leq q,r\leq \infty$. We say
\[F\dot{B}^s_{p,q,(r)}(\RR^d)=\Big\{f\in \mathcal{S}'(\RR^d)\,\big|\,\|f\|_{F\dot{B}^s_{p,q,(r)}}=\Big\|\big\{2^{js}\|\widehat{\dot{\Delta}_j f}\|_{L^{p,r}}\big\}\Big\|_{\ell^q}<\infty\Big\}\]
\end{definition}
 According to the property of Lorentz space, we see that $F\dot{B}^{s}_{p,q,(p)}(\RR^d)$ coincides with $F\dot{B}^s_{p,q}(\RR^d)$. In addition,
\begin{equation}\label{fh-embd}
F\dot{B}^{s}_{p,q}(\RR^d)\hookrightarrow F\dot{B}^{s}_{p,q,(r)}(\RR^d)\, \text{ if }\, r> p,\quad  F\dot{B}^{s}_{p,q,(r)}(\RR^d)\hookrightarrow F\dot{B}^{s}_{p,q}(\RR^d)\, \text{ if }\, r< p.
\end{equation}
With this definition, we will establish the following interpolation theory in Fourier-Herz spaces with the help of  $K$-function.
\begin{lemma}\label{lem.inter}
Let $ s_1,\,s_2\in\RR$, $1<p_1,p_2<\infty$, $1\leq q_1,q_2<\infty$, $p_1\neq p_2$, and $0<\theta<1$. Then
\[\Big(F\dot{B}^{s_1}_{p_1,q_1}(\RR^d),F\dot{B}^{s_2}_{p_2,q_2}(\RR^d)\Big)_{\theta,q}=F\dot{B}^s_{p,q,(q)}(\RR^d)\]
where $s_1(1-\theta)+s_2\theta=s$, ${1-\theta\over p_1}+{\theta\over p_2}={1\over p}$ and ${1-\theta\over q_1}+{\theta\over q_2}={1\over q}$.

Besides, for any $f\in F\dot{B}^s_{p,q,(q)}$, we have
\[\|f\|_{F\dot{B}^s_{p,q,(q)}}=\bigg(\sum_{j\in\ZZ}2^{jq\theta}K(f,j)^q\bigg)^{1\over q}\]
with
\[K(f,j)=\inf_{f=g+h}\Big(\|g\|_{F\dot{B}^{s_1}_{p_1,q_1}}+2^{-j}
\|h\|_{F\dot{B}^{s_2}_{p_2,q_2}}\Big).\]
\end{lemma}
\begin{proof}
We here sketch key points  of the proof, because it is similar to the proof of the interpolation in Besov spaces in \cite{Tr78}.

From Theorem 2.4.1/(c) in \cite{Tr78}, we have that
\[\big(\ell^{q_1}(A_j),\ell^{q_2}(B_j)\big)_{\theta,q}=\ell^{q}\big((A_j,B_j)_{\theta,q}\big).\]
Let $A_j=2^{js_1}FL^{p_1}(\RR^d)$ and $B_j=2^{js_2}FL^{p_2}(\RR^d)$. Thus, we get the required result.
\end{proof}
In view of Lemma \ref{lem.inter}, we can get the following  decomposition  in  $F\dot{B}^{s_p}_{p,q}(\RR^3)$ with $s_p=2-\frac{3}{p}$.
\begin{lemma}\label{lem.decom}
Let solenoidal vector field $f\in F\dot{B}^{s_p}_{p,q}(\RR^3)$ with $p\in(1,{3\over 2})$ and $ q\in [1,\infty)$, then for each $j\in Z$, there exist solenoidal vector fields $g_0\in F\dot{B}^{s}_{\tilde{p},\tilde{q}}$, $h_0\in L^2$ such that
\[f=g_0+h_0,\]
\[\|g_0\|_{F\dot{B}^{s}_{\tilde{p},\tilde{q}}}\leq C2^{-j\theta}\|f\|_{F\dot{B}^{s_p}_{p,q}}\quad\text{ and }\quad\|h_0\|_{L^2}\leq C2^{j(1-\theta)}\|f\|_{F\dot{B}^{s_p}_{p,q}} ,\]
where $C$ is a absolute constant and $\theta,\,s_p,\,s,\,p,\,\tilde{p},\,q,\,\tilde{q}$ satisfy the compatibility condition:
\begin{equation}\label{RC}
\begin{split}
\theta\in (0,1), \quad &s_p=(1-\theta)s,\quad\frac{1}{p}=\frac{\theta}{2}+\frac{1-\theta}{ \tilde{p}},\quad \tilde{p}\in (1,\frac32)\\
&\frac{1}{\tilde{q}}=\begin{cases}\frac{1}{1-\theta }\left(\frac{1}{ q}-\frac{\theta}{2}\right) \quad& q<p;\\
 \frac{1}{\tilde{p}}  \quad& q\geq p.
\end{cases}
\end{split}
\end{equation}
\end{lemma}
\begin{proof}
We  proceed the lemma  in two cases: $p\geq q$ and $p<q$.
\medskip

\noindent Case 1: $p\geq q$.
By Lemma \ref{lem.inter}, we have \[F\dot{B}^{s_p}_{p,q}(\RR^3)\hookrightarrow F\dot{B}^{s_p}_{p,p}(\RR^3)=\Big(F\dot{B}^{s}_{\tilde{p},\tilde{p}}(\RR^3),L^2(\RR^3)\Big)_{\theta,p}\]
where $\frac{1-\theta}{\tilde{p}}+\frac{\theta}{2}=\frac{1}{p}$ and $s(1-\theta)=s_p$.
Thus  we have for any $f\in F\dot{B}^{s_p}_{p,q}(\RR^3)$,
\begin{align*}
\|f\|_{F\dot{B}^{s_p}_{p,p}}=\bigg(\sum_{j\in\ZZ}2^{jp\theta}\Big(\inf_{f=g+h}
\big(\|g\|_{F\dot{B}^{s}_{\tilde{p},\tilde{p}}}+2^{-j}\|h\|_{L^2}\big)\Big)^p\bigg)^{1\over p}.
\end{align*}
From this equality, for any $j\in \ZZ$, there exist $\tilde{g}_0,\,\tilde{h}_0$ such that $f=\tilde{g}_0+\tilde{h}_0$ and
\begin{align*}
\|\tilde{g}_0\|_{F\dot{B}^{s}_{\tilde{p},\tilde{p}}}+2^{-j}\|\tilde{h}_0\|_{L^2}\leq &2\inf_{f=g+h}
\big(\|g\|_{F\dot{B}^{s}_{\tilde{p},\tilde{p}}}+2^{-j}\|h\|_{L^2}\big)\\
\leq &2^{1-j\theta}\|f\|_{F\dot{B}^{s_p}_{p,p}}\leq C2^{-j\theta}\|f\|_{F\dot{B}^{s_p}_{p,q}}.
\end{align*}
Letting $g_0=\mathbb{P}(\tilde{g}_0)$ and $h_0=\mathbb{P}(\tilde{h}_0)$, we have by the fact $\Div f=0$ that  $$f=g_0+h_0.$$
In terms of Calder\'{o}n-Zygmund estimates, we readily get
\begin{equation*}
\|g_0\|_{F\dot{B}^{s}_{\tilde{p},\tilde{p}}}\leq C\|\tilde{g}_0\|_{F\dot{B}^{s}_{\tilde{p},\tilde{p}}}\leq C2^{-j\theta}\|f\|_{F\dot{B}^{s_p}_{p,q}}\end{equation*}
 and
 \begin{equation*}
 \|h_0\|_{L^2}\leq C\|\tilde{h}_0\|_{L^2}\leq C2^{j(1-\theta)}\|f\|_{F\dot{B}^{s_p}_{p,q}}.
\end{equation*}
Case 2: $p<q$.  According to \eqref{fh-embd} and Lemma \ref{lem.inter}, we have
\[F\dot{B}^{s_p}_{p,q}(\RR^3)\hookrightarrow F\dot{B}^{s_p}_{p,q,(q)}(\RR^3)=\big(F\dot{B}^{s}_{\tilde{p},\tilde{q}}(\RR^3),L^2(\RR^3)\big)_{\theta,q}\]
with
$$\frac{1-\theta}{\tilde{p}}+\frac{\theta}{2}=\frac{1}{p},\quad \frac{1-\theta}{\tilde{q}}+\frac{\theta}{2}=\frac{1}{q},\quad s(1-\theta)=s_p,\quad 0<\theta<1.$$
Thus, for any $j\in \ZZ$, there exist $\tilde{g}_0,\,\tilde{h}_0$ such that $f=\tilde{g}_0+\tilde{h}_0$ and
\begin{align*}\|\tilde{g}_0\|_{F\dot{B}^{s}_{\tilde{p},\tilde{q}}}+2^{-j}\|\tilde{h}_0\|_{L^2}\leq & 2\inf_{f=g+h}\big(\|g\|_{F\dot{B}^{s}_{\tilde{p},\tilde{q}}}+2^{-j}\|h\|_{L^2}\big)\\ \leq & 2^{1-j\theta}\|f\|_{F\dot{B}^{s_p}_{p,q,(q)}}\leq C2^{-j\theta}\|f\|_{F\dot{B}^{s_p}_{p,q}}.\end{align*}
Then, letting $g_0=\mathbb{P}(\tilde{g}_0)$ and $h_0=\mathbb{P}(\tilde{h}_0)$,  and using  the above inequality and $\Div f=0$, we eventually obtain that $f=g_0+h_0$, where
\[\|g_0\|_{F\dot{B}^{s}_{\tilde{p},\tilde{q}}}\leq C2^{-j\theta}\|f\|_{F\dot{B}^{s_p}_{p,q}}\quad\text{ and }\quad\|h_0\|_{L^2}\leq C2^{j(1-\theta)}\|f\|_{F\dot{B}^{s_p}_{p,q}}.\]
Thus we end the proof of Lemma \ref{lem.decom}.
\end{proof}
Finally we review Banach fixed's theorem and the associated propagation of regularity. Their proof can be found in \cite{GIP03}.
\begin{lemma}[\cite{GIP03}]\label{fixed}
Let $(X,\|\cdot\|)$ be an abstract Banach space, $L:\,X\to X$ be a linear bounded operator such that for a constant $\lambda\in[0,1)$, we have
$\|L(x)\|\leq \lambda\|x\|$ for all $x\in X,$
and $B :X \times X \to X$ be a bilinear mapping such that
$$\|B(x_1,x_2)\|\le \eta\|x_1\|\,\|x_2\|\qquad \forall\,x_1,\, x_2\in X $$
for some $\eta>0$. Then, for every $y\in X$ satisfying $4\eta\|y\|< (1-\lambda)^{2}$, the equation
\begin{equation}\label{eq.fixed-theorem}
x = y+L(x) +B(x,x)
\end{equation}
 has a solution $x\in X$.

 In particular, this solution satisfies $\|x\|\le \frac{2\|y\|}{1-\lambda}$,
and it is the only one among all solutions satisfying $\|x\|< \frac{1-\lambda}{2\eta}$.
\end{lemma}

\begin{lemma}[\cite{GIP03}]\label{regul}
In the notation of Lemma \ref{fixed}, let $E$ be a Banach space. Suppose $L:\,E\to E$ be a linear bounded operator such that for a $\beta\in[0,1)$, we have
$\|L(x)\|\leq \beta\|x\|$ for all $x\in E,$
and $B: X \times E \to E$ and $E \times X \to E$ be a bilinear mapping such that
$$\max\big\{\|B(x_1,x_2)\|_{E},\|B(x_2,x_1)\|_{E}\big\}\le \gamma\|x_1\|_{X}\,\|x_2\|_{E}\quad\text{for every}\quad x_1\in X,\, x_2\in E $$
for some $\gamma>0$. Then, if $\beta\gamma\leq\lambda\eta$, for all $y\in E$ satisfying $4\eta\|y\|_{X}<(1-\lambda)^2$, the solution given in Lemma \ref{fixed} belongs to $E$ and satisfies $\|x\|_{E}\leq 2\beta \|a\|_{E}$.
\end{lemma}

\section{Mild solution of the 3D Navier-Stokes system \eqref{NS}  in  subcritical and critical framework}\label{sec.3}
\setcounter{section}{3}\setcounter{equation}{0}

In this section, we will mainly study the well-posed theory of problem \eqref{NS} associated with initial data in subcritical Fourier-Herz spaces $F\dot{B}^{s}_{p,q}(\RR^3)$.

First of all, we recall the well-posed theory of problem \eqref{NS} in the framework of critical Fourier-Herz space. In \cite{LZ17}, J. Li and X. Zheng obtain the following well-posed results to~\eqref{NS} in critical Fourier-Herz space $F\dot{B}^{s_p}_{p,q}(\RR^3)$ with $1\leq p,q\leq \infty$:
\begin{theorem}[\cite{LZ17}]\label{well-cri}
Let $u_0\in F\dot{B}^{s_p}_{p,q}(\RR^3)$ satisfying $\Div u_0=0$, $1\leq p,q\leq \infty$ and $1\leq q\leq 2$ if $p=1$.
\begin{enumerate}
  \item[\rm (1)] For $1\leq q<\infty$, there exists a $T^*=T^*(u_0)>0$ such that problem \eqref{NS} has a unique local in time mild solution $u$ satisfying
  $$u \in C\big([0,T^*);\,F\dot{B}^{s_p}_{p,q}(\RR^3)\big) \cap \widetilde{L}^r_{\rm loc}\big([0,T^*);\,F\dot{B}^{2+s_p}_{p,q}(\RR^3)\big) \quad\forall r\in [1,\infty].$$
  Moreover we have that
  \begin{equation}\label{blow}
  \lim_{\delta\to0+}\|u\|_{\widetilde{L}^r([0,T^*-\delta];\,F\dot{B}^{\frac2r+s_p}_{p,q})}=\infty\Longleftrightarrow \,T^*<\infty.
  \end{equation}
  \item [\rm (2)] For $1\leq q\leq \infty$, if there exists a positive constant $\eta_0$ such that $\|u_0\|_{F\dot{B}^{s_p}_{p,q}}<\eta_0$, then problem~\eqref{NS} has a unique global in time mild solution $u$ satisfying
  $$u \in C([0,\infty);\,F\dot{B}^{s_p}_{p,q}(\RR^3)) \cap \widetilde{L}^r([0,\infty);\,F\dot{B}^{2+s_p}_{p,q}(\RR^3)) \quad\forall r\in [1,\infty].$$
  \end{enumerate}
\end{theorem}
Next, we turn to show the local well-posedness of \eqref{NS} in subcritical spaces $F\dot{B}^{s}_{p,q}(\RR^3)$ with $s\in(s_p,0]$.
\begin{theorem}\label{well-sub}
	For any $u_0\in F\dot{B}^{s}_{p,q}(\RR^3)$ with $(s,p)\in \{[s_p,0]\times[1,\infty]\}\setminus (0,1)$ and $q\in [1,\infty]$, satisfying $\Div u_0=0$, there exists a $T^*=T^*(\nu,\|u_0\|_{F\dot{B}^s_{p,q}})>0$ such that system \eqref{NS} has a unique local in time mild solution $u$ satisfying
	$$u \in X_T\triangleq C([0,T);\,F\dot{B}^{s}_{p,q}(\RR^3)) \cap \widetilde{L}^r_{\rm loc}([0,T^*);\,F\dot{B}^{\frac2r+s}_{p,q}(\RR^3)) \quad\forall r\in [1,\infty].$$
	Moreover, $T^*$ and $u$ satisfy
	\[T^*\leq C(\nu)\|u_0\|_{F\dot{B}^s_{p,q}}^{-\frac 2{s-s_p}}\quad\text{and}\quad \|u\|_{X_T}\leq C(\nu)\|u_0\|_{F\dot{B}^s_{p,q}}.\]
\end{theorem}
\begin{proof}
	By Duhamel formula, one writes equation \eqref{NS} in the integral form
	\begin{equation}\label{eq.Inter}
	u(x,t)=e^{\nu t\Delta}u_0+\int^t_0 e^{\nu(t-\tau)\Delta}\mathbb{P}(u\cdot\nabla u)\,\mathrm{d}\tau\triangleq G+B(u,u).
	\end{equation}
	Now we are going to apply Lemma \ref{fixed} to get the local well-posedness of problem  \eqref{NS} in subcritical spaces $F\dot{B}^{s}_{p,q}(\RR^3)$ with $s\in(s_p,0]$.
	
From estimate  \eqref{est.TS}, one has
	\begin{equation}\label{est.TS-SS}
	\|e^{t\Delta}u_0\|_{\widetilde{L}^r([0,T];\,F\dot{B}^s_{p,q})}\leq C\|u_0\|_{F\dot{B}^s_{p,q}}\quad \text{for all}\quad s\in \RR,\, 1\leq r,p,q\leq \infty.
\end{equation}
For the bilinear term $B(u,v)$,
 in terms of Bony-paraproduct decomposition, one writes
 $$B(u,v)=I+II+III,$$ where
\[I\triangleq\sum_{k\in \ZZ}\int^t_0 e^{-\nu(t-\tau)\Delta}\mathbb{P}(\dot{S}_ku \otimes\dot{\Delta}_k v)\,\mathrm{d}\tau,\qquad II\triangleq\sum_{k\in \ZZ}\int^t_0 e^{-\nu(t-\tau)\Delta}\mathbb{P}(\dot{S}_k v\otimes\dot{\Delta}_k u)\,\mathrm{d}\tau\]
\[III\triangleq \sum_{k\in \ZZ}\int^t_0 e^{-\nu(t-\tau)\Delta}\mathbb{P}(\dot{\Delta}_ku\otimes\tilde{\dot{\Delta}}_k v)\,\mathrm{d}\tau.\]
For $I$, by H\"{o}lder's and Young's inequalities, we have
\begin{align*}
&2^{js}\|\dot{\Delta}_jI\|_{L^{\infty}([0,T];\,FL^p)}\\
\lesssim_{\nu}&\int^T_0e^{-c\nu t2^{2j}}2^{j(4-s-\frac 3p)}\,\mathrm{d}t\sum_{|j-k|\leq 2}2^{(j-k)(2s-3+\frac 3p)}2^{k(s-3+\frac 3p)}\|\dot{S}_{k-1} u\|_{L^{\infty}([0,T];\,FL^1)}\\
&\times2^{ks}\|\dot{\Delta}_{k} v\|_{L^{\infty}([0,T];\,FL^p)}\\
\lesssim_{\nu}&T^{\frac{s-s_p}2}\sum_{|j-k|\leq 2}2^{(j-k)(2s-3+\frac 3p)}2^{k(s-3+\frac 3p)}\|\dot{S}_{k-1} u\|_{L^{\infty}([0,T];\,FL^1)}2^{ks}\|\dot{\Delta}_{k} v\|_{L^{\infty}([0,T];\,FL^p)}.
\end{align*}
Taking $\ell^q(\ZZ)$-norm on the above inequality, we get
\begin{equation}\label{eq.well-1}
\|I\|_{\widetilde{L}^{\infty}([0,T];\,F\dot{B}^s_{p,q})}\lesssim_{\nu}T^{\frac{s-s_p}2}\|u\|_{\widetilde{L}^{\infty}([0,T];\,F\dot{B}^s_{p,q})}
\|v\|_{\widetilde{L}^{\infty}([0,T];\,F\dot{B}^s_{p,q})}.
\end{equation}
Moreover, we have
\begin{align*}
&2^{j(s+2)}\|\dot{\Delta}_jI\|_{L^1([0,T];\,FL^p)}\\
\lesssim_{\nu}& \int^t_0 e^{-c\nu t2^{2j}}2^{j(4-s-\frac 3p)}\,\mathrm{d}t \sum_{|j-k|\leq 2}2^{(j-k)(2s-1+\frac 3p)}2^{k(s-3+\frac 3p)}\|\dot{S}_{k-1} u\|_{L^{\infty}([0,T];\,FL^1)}\\
&\times2^{k(s+2)}\|\dot{\Delta}_{k} v\|_{L^1([0,T];\,FL^p)}\\
\lesssim_{\nu}& T^{\frac{s-s_p}2}\sum_{|j-k|\leq 2}2^{(j-k)(2s-1+\frac 3p)}2^{k(s-3+\frac 3p)}\|\dot{S}_{k-1} u\|_{L^{\infty}([0,T];\,FL^1)}2^{k(s+2)}\|\dot{\Delta}_{k} v\|_{L^1([0,T];\,FL^p)}.
\end{align*}
Hence, we have
\begin{equation}\label{eq.well-1'}
\|I\|_{\widetilde{L}^1([0,T];\,F\dot{B}^{s+2}_{p,q})}\lesssim_{\nu}T^{\frac{s-s_p}2}\|u\|_{\widetilde{L}^{\infty}([0,T];\,F\dot{B}^s_{p,q})}
\|v\|_{\widetilde{L}^{1}([0,T];\,F\dot{B}^{s+2}_{p,q})}.
\end{equation}
Similarly, we can show
\begin{equation}\label{eq.well-2}
\|II\|_{\widetilde{L}^{\infty}([0,T];\,F\dot{B}^s_{p,q})}\lesssim_{\nu}T^{\frac{s-s_p}2}
\|u\|_{\widetilde{L}^{\infty}([0,T];\,F\dot{B}^s_{p,q})}
\|v\|_{\widetilde{L}^{\infty}([0,T];\,F\dot{B}^s_{p,q})}
\end{equation}
and
\begin{equation}\label{eq.well-2'}
\|II\|_{\widetilde{L}^1([0,T];\,F\dot{B}^{s+2}_{p,q})}\lesssim_{\nu}T^{\frac{s-s_p}2}
\|u\|_{\widetilde{L}^1([0,T];\,F\dot{B}^{s+2}_{p,q})}
\|v\|_{\widetilde{L}^{\infty}([0,T];\,F\dot{B}^s_{p,q})}.
\end{equation}
For the remainder term $III$, for any $r>2$ satisfying $-s<\frac 2r<\frac {s_p-s}2+1$, we have by H\"older's and Young's inequalities that
\begin{align*}
&2^{js}\|\dot{\Delta}_jIII\|_{L^{\infty}([0,T];\,FL^p)}\\
\lesssim_{\nu}&\Big(\int^T_0 e^{-c\nu t2^{2j}\frac r{r-2}}2^{j(4-s-\frac 3p-\frac 4r)\frac r{r-2}}\,\mathrm{d}t\Big)^{\frac {r-2}r}\sum_{k>j-2}\Big(2^{(j-k)(2s+\frac 4r)}2^{k(s+\frac 2r)}\|\dot{\Delta}_{k} u\|_{L^r([0,T];\,FL^p)}\\
&\qquad\qquad\qquad\qquad\qquad\qquad\qquad\qquad\qquad\quad\,\times 2^{k(s+\frac 2r)}\|\tilde{\dot{\Delta}}_{k} v\|_{L^r([0,T];\,FL^p)}\Big)\\
\lesssim_{\nu}& \Big(\int^T_0 t^{-1+\frac{s-s_p}2\frac{r}{r-2}}\,\mathrm{d}t\Big)^{1-\frac2r}\sum_{k>j- 2}\Big(2^{(j-k)(2s+\frac 4r)}2^{k(s+\frac 2r)}\|\dot{\Delta}_{k} u\|_{L^r([0,T];\,FL^p)}\\
&\qquad\qquad\qquad\qquad\qquad\qquad\qquad\times2^{k(s+\frac 2r)}\|\tilde{\dot{\Delta}}_{k} v\|_{L^r([0,T];\,FL^p)}\Big)\\
\lesssim_{\nu} &T^{\frac{s-s_p}2}\sum_{k>j-2}2^{(j-k)(2s+\frac 4r)}\|\dot{\Delta}_{k} u\|_{L^r([0,T];\,FL^p)}2^{k(s+\frac 2r)}\|\tilde{\dot{\Delta}}_{k} v\|_{L^r([0,T];\,FL^p)}.
\end{align*}
So, by the interpolation and Young inequalities, we get
\begin{equation}\label{eq.well-3}
\begin{split}
&\|III\|_{\widetilde{L}^{\infty}([0,T];\,F\dot{B}^s_{p,q})}\\
\lesssim_{\nu}& T^{\frac{s-s_p}2}\|u\|_{\widetilde{L}^{r}([0,T];\,F\dot{B}^{s+\frac 2r}_{p,q})}
\|v\|_{\widetilde{L}^{r}([0,T];\,F\dot{B}^{s+\frac 2r}_{p,q})}\\
\lesssim_{\nu}&T^{\frac{s-s_p}2}\|u\|_{\widetilde{L}^{\infty}([0,T];\,F\dot{B}^s_{p,q})
	\cap\widetilde{L}^1([0,T];\,F\dot{B}^{s+2}_{p,q})}
\|v\|_{\widetilde{L}^{\infty}([0,T];\,F\dot{B}^s_{p,q})\cap\widetilde{L}^1([0,T];\,F\dot{B}^{s+2}_{p,q})}.
\end{split}
\end{equation}
Furthermore, we have
\begin{align*}
&2^{j(s+2)}\|\dot{\Delta}_jIII\|_{L^1([0,T];\,FL^p)}\\
\lesssim_{\nu}& \int^T_0 e^{-c\nu t2^{2j}}2^{j(4-s-\frac 3p)}\,\mathrm{d}t\sum_{k>j-2}2^{(j-k)(2s+2)}2^{ks}\|\dot{\Delta}_{k} u\|_{L^{\infty}([0,T];\,FL^p)}2^{k(s+2)}\|\tilde{\dot{\Delta}}_{k} v\|_{L^1([0,T];\,FL^p)}\\
\lesssim_{\nu}&T^{\frac{s-s_p}2}\sum_{k>j-2}2^{(j-k)(2s+2)}2^{ks}\|\dot{\Delta}_{k} u\|_{L^{\infty}([0,T];\,FL^p)}2^{k(s+2)}\|\tilde{\dot{\Delta}}_{k} v\|_{L^1([0,T];\,FL^p)}.
\end{align*}
From this, we obtain that
\begin{equation}\label{eq.well-3'}
\|III\|_{\widetilde{L}^1([0,T];\,F\dot{B}^{s+2}_{p,q})}\lesssim_{\nu}T^{\frac{s-s_p}2}\|u\|_{\widetilde{L}^{\infty}([0,T];\,F\dot{B}^s_{p,q})}
\|v\|_{\widetilde{L}^1([0,T];\,F\dot{B}^{s+2}_{p,q})}.
\end{equation}
Collecting \eqref{eq.well-1}-\eqref{eq.well-3'}, we finally obtain that
\begin{align*}
&\|B(u,v)\|_{\widetilde{L}^{\infty}([0,T];F\dot{B}^s_{p,q})\cap \widetilde{L}^1([0,T];F\dot{B}^{s+2}_{p,q})}\\
\lesssim_{\nu}&T^{\frac{s-s_p}2}\|u\|_{\widetilde{L}^{\infty}([0,T];F\dot{B}^s_{p,q})\cap \widetilde{L}^1([0,T];F\dot{B}^{s+2}_{p,q})}\|v\|_{\widetilde{L}^{\infty}([0,T];F\dot{B}^s_{p,q})\cap \widetilde{L}^1([0,T];F\dot{B}^{s+2}_{p,q})}.
\end{align*}
By choosing a suitable $T>0$,  we obtain  that there exists a unique local mild solution $u\in X_T$ of problem \eqref{NS} in subcritical spaces $F\dot{B}^{s}_{p,q}(\RR^3)$ with  $s\in(s_p,0]$ via the Banach fixed point theorem.
 \end{proof}
Next, we will give the analyticity and decay estimates for the solution of the mild solution   to problem \eqref{NS}  constructed by Theorem \ref{well-cri}  and  Theorem \ref{well-sub}.
\begin{proposition}\label{pro.ana}
Let $ s\in [s_p ,0]$. Assume that $u$ is the mild solution on $\RR^3\times [0,T^*)$ to \eqref{NS}  constructed by Theorem \ref{well-cri} or Theorem \ref{well-sub}. Then,
\begin{equation}\label{eq.ana-1}
\|tu(t)\|_{\widetilde{L}^{\infty}([0,T];F\dot{B}^{s+2}_{p,q})}<\infty,
\end{equation}
for each $T<T^*.$
\end{proposition}
\begin{proof}
Set
$$Y_T\triangleq \widetilde{L}^{\infty}([0,T];\,F\dot{B}^s_{p,q}(\RR^3))\cap\widetilde{L}^1([0,T];\,
F\dot{B}^{s+2}_{p,q}(\RR^3)).$$
We to do To do this, we need to show that:
\begin{enumerate}[\rm (i)]
\item there exists a constant $C(\nu)>0$ such that
\begin{align*}
&\max\left\{ \|tB(u,v)\|_{\widetilde{L}^{\infty}([0,T];\,F\dot{B}^{s+2}_{p,q})},
\|tB(v,u)\|_{\widetilde{L}^{\infty}([0,T];\,F\dot{B}^{s+2}_{p,q})}\right\}\\
\leq& C(\nu) T^{\frac{s-s_p}2}\|u\|_{Y_{T}}\|tv\|_{\widetilde{L}^{\infty}([0,T];\,F\dot{B}^{s+2}_{p,q})}+C(\nu) T^{\frac{s-s_p}2}\|u\|_{Y_{T}}\|v\|_{Y_{T}}.
\end{align*}
\item $te^{t\Delta}u_0\in\tilde{L}^{\infty}([0,T];\,F\dot{B}^{s+2}_{p,q})$.
\end{enumerate}
Let us start by estimating $B(u,v)$.
By Holder's and Young's inequality, the term $I$ can be bounded as follows
\begin{align*}
&2^{j(s+1)}t\|\dot{\Delta}_jI\|_{FL^p}\\
\lesssim& t\sum_{|j-k|\leq 2}2^{(j-k)(2s-3+\frac3p+\frac2r)}
2^{j(-s+6-\frac3p-\frac2r)}\Big(\int^{t/2}_0 e^{-c\nu(t-\tau)2^{2j}}\tau^{-\frac1r}\,\mathrm{d}\tau+
\int^{t}_{t/2} e^{-c\nu(t-\tau)2^{2j}}\tau^{-\frac1r}\,\mathrm{d}\tau\Big) \\
&\times 2^{k(s-3+\frac3p)}\|\dot{S}_{k-1}u\|_{L^{\infty}([0,T_1];\,FL^1)}2^{k(s+\frac2r)}\|t^{\frac 1 r}\dot{\Delta}_k v\|_{L^{\infty}([0,T_1];\,FL^p)}\\
\lesssim&_{\nu}t\sum_{|j-k|\leq 2}2^{(j-k)(2s-3+\frac3p+\frac2r)}
\big(2^{j(-s+6-\frac3p-\frac2r)}e^{-c\nu t2^{2j}}t^{1-\frac1r}+
2^{j(-s+4-\frac3p-\frac2r)}e^{-c\nu t2^{2j}}t^{-\frac1r}\big) \\
&\times 2^{k(s-3+\frac3p)}\|\dot{S}_{k-1}u\|_{L^{\infty}([0,T];\,FL^1)}2^{k(s+\frac2r)}\|t^{\frac 1 r}\dot{\Delta}_k v\|_{L^{\infty}([0,T];\,FL^p)}.
\end{align*}
Choosing  $r\in (\frac{2p}{3p-3},\infty)$ such that $-s+4-\frac3p-\frac2r>0$ and $-s+6-\frac3p-\frac2r>0$, we have
\begin{align*}
&2^{j(s+1)}t\|\dot{\Delta}_jI\|_{FL^p}\\
\lesssim&_{\nu} t^{\frac{s-s_p}2}\sum_{|j-k|\leq 2}2^{(j-k)(2s-3+\frac3p+\frac2r)}
2^{k(s-3+\frac3p)}\|\dot{S}_{k-1}u\|_{L^{\infty}([0,T];\,FL^1)}2^{k(s+\frac2r)}\|t^{\frac1r}\dot{\Delta}_k v\|_{L^{\infty}([0,T];\,FL^p)}.
\end{align*}
Taking $L^{\infty}[0,T]$ and then $\ell^q(\ZZ)$-norm on the above inequality, we obtain that
\[\|tI\|_{\widetilde{L}^{\infty}([0,T];\,F\dot{B}^{s+2}_{p,q})}\lesssim_{\nu}T^{\frac{s-s_p}2}\|u\|_{\widetilde{L}^{\infty}([0,T];\,F\dot{B}^{s}_{p,q})}
\|t^{\frac 1r}v\|_{\widetilde{L}^{\infty}([0,T];\,F\dot{B}^{s+\frac 1r}_{p,q})}.\]
By the interpolation and Young's inequalities, we get
\begin{equation}\label{eq.ana1}
\begin{split}
\|tI\|_{\widetilde{L}^{\infty}([0,T];\,F\dot{B}^{s+2}_{p,q})}
\lesssim&_{\nu}T^{\frac{s-s_p}2}\|u\|_{\widetilde{L}^{\infty}([0,T];\,F\dot{B}^{s}_{p,q})}
\|tv\|^{1/r}_{\widetilde{L}^{\infty}([0,T];\,F\dot{B}^{s+2}_{p,q})}
\|v\|^{1-1/r}_{\widetilde{L}^{\infty}([0,T];\,F\dot{B}^{s}_{p,q})}\\
\lesssim &_{\nu}T^{\frac{s-s_p}2}\|u\|_{\widetilde{L}^{\infty}([0,T];\,F\dot{B}^{s}_{p,q})}
\|tv\|_{\widetilde{L}^{\infty}([0,T];\,F\dot{B}^{s+2}_{p,q})}\\
&+T^{\frac{s-s_p}2}\|u\|_{\widetilde{L}^{\infty}([0,T];\,F\dot{B}^{s}_{p,q})}
\|v\|_{\widetilde{L}^{\infty}([0,T];\,F\dot{B}^{s}_{p,q})}.
\end{split}
\end{equation}
For $II$, similar to $I$, for any $r\in (\frac{2p}{3p-3},\infty)$, we have
\begin{align*}
&2^{j(s+1)}t\|\dot{\Delta}_jII\|_{FL^p}\\
\lesssim&_{\nu} t^{\frac{s-s_p}2}\sum_{|j-k|\leq 2}2^{(j-k)(2s-3+\frac3p+\frac2r)}
2^{k(s-3+\frac3p+\frac2r)}\| t^{\frac1r}\dot{S}_{k-1}v\|_{L^{\infty}([0,T];\,FL^1)}2^{ks}\|\dot{\Delta}_k v\|_{L^{\infty}([0,T];\,FL^p)}.
\end{align*}
Since $r\in (\frac{2p}{3p-3},\infty)$, which ensures $s-3+\frac3p+\frac2r<0$,
taking $L^{\infty}[0,T]$ and then $\ell^q(\ZZ)$ on above inequality, by interpolation inequality and Young's inequality, we get
\begin{equation}\label{eq.ana2}
\begin{split}
\|t II\|_{\widetilde{L}^{\infty}([0,T];\,F\dot{B}^{s+2}_{p,q})}\lesssim&_{\nu} T^{\frac{s-s_p}2}\|u\|_{\widetilde{L}^{\infty}([0,T];\,F\dot{B}^{s}_{p,q})}
\|tv\|^{1/r}_{\widetilde{L}^{\infty}([0,T];\,F\dot{B}^{s+2}_{p,q})}
\|v\|^{1-1/r}_{\widetilde{L}^{\infty}([0,T];\,F\dot{B}^{s}_{p,q})}\\
\lesssim &_{\nu}T^{\frac{s-s_p}2}\|u\|_{\widetilde{L}^{\infty}([0,T];\,F\dot{B}^{s}_{p,q})}
\|tv\|_{\widetilde{L}^{\infty}([0,T];\,F\dot{B}^{s+2}_{p,q})}\\
&+T^{\frac{s-s_p}2}\|u\|_{\widetilde{L}^{\infty}([0,T];\,F\dot{B}^{s}_{p,q})}
\|v\|_{\widetilde{L}^{\infty}([0,T];\,F\dot{B}^{s}_{p,q})}.
\end{split}
\end{equation}
For the remainder term $III$, by H\"{o}lder's and Young's inequalities, for any $r\in(\frac{2p}{3p-3},\infty)$ and $r_1>\frac{r}{r-1}$, we have
\begin{align*}
&2^{j(s+2)}t\|\dot{\Delta}_jIII\|_{FL^p}\\
\lesssim_{\nu}& t\sum_{|j-k|\leq 2}2^{(j-k)(2s+\frac2r+\frac2{r_1})}2^{j(6-\frac3p-s-\frac2r-\frac2{r_1})}
\bigg(\int^t_0 e^{-c\nu(t-\tau)2^{2j}\frac{r_1}{r_1-1}}\tau^{-\frac{r_1}{(r_1-1)r}}\,\mathrm{d}\tau\bigg)^{\frac{r_1-1}{r_1}}\\
&\times2^{k(s+\frac2r)}\|t^{\frac1r}\tilde{\dot{\Delta}}_{k}v\|_{L^{\infty}([0,T];\,FL^p)}
2^{k(s+\frac2{r_1})}\|\dot{\Delta}_k u\|_{L^{r_1}([0,T];\,FL^p)}.
\end{align*}
Due to $\frac{r_1}{r_1-1}<r$, we have that
\begin{align*}
&\Big(\int^t_0 e^{-c\nu(t-\tau)2^{2j}\frac{r_1}{r_1-1}}\tau^{-\frac{r_1}{(r_1-1)r}}\,\mathrm{d}\tau\Big)^{\frac{r_1-1}{r_1}}\\
\leq &\Big(\int^{\frac t2}_0 e^{-c\nu(t-\tau)2^{2j}\frac{r_1}{r_1-1}}\tau^{-\frac{r_1}{(r_1-1)r}}\,\mathrm{d}\tau\Big)^{\frac{r_1-1}{r_1}}
+\Big(\int^t_{\frac t2} e^{-c\nu(t-\tau)2^{2j}\frac{r_1}{r_1-1}}\tau^{-\frac{r_1}{(r_1-1)r}}\,\mathrm{d}\tau\Big)^{\frac{r_1-1}{r_1}}\\
\leq&C e^{-\frac c2\nu t2^{2j}}t^{1-\frac{1}{r}-\frac{1}{r_1}}+C2^{-j(2-\frac{2}{r_1})}e^{-c2\nu t2^{2j}}t^{-\frac{1}{r}}.
\end{align*}
Substituting this estimate into the above inequality, thanks to that $6-s-\frac3p+\frac2r+\frac2{r_1}>0$ and $4-s-\frac3p-\frac2r>0$, we can get that
\begin{align*}
&2^{j(s+2)}t\|\dot{\Delta}_jIII\|_{FL^p}\\
\lesssim&_{\nu} t\sum_{k>j-2}2^{(j-k)(2s+\frac2r+\frac2{r_1})}
\big(2^{j(6-\frac3p-s-\frac2r-\frac2{r_1})}e^{-\frac c2\nu t2^{2j}}t^{1-\frac{1}{r}-\frac{1}{r_1}}+2^{j(4-\frac3p-s-\frac2r)}e^{-c2\nu t2^{2j}}t^{-\frac{1}{r}}\big)\\
&\times2^{k(s+\frac2r)}\|t^{\frac1r}\tilde{\dot{\Delta}}_{k}v\|_{L^{\infty}([0,T];\,FL^1)}
2^{k(s+\frac2{r_1})}\|\dot{\Delta}_k u\|_{L^{r_1}([0,T];\,FL^p)}.\\
\lesssim&_{\nu} t^{\frac{s-s_p}2}\sum_{k>j-2}2^{(j-k)(2s+\frac2r+\frac2{r_1})}
2^{k(s+\frac2r)}\|t^{\frac1r}\tilde{\dot{\Delta}}_{k}v\|_{L^{\infty}([0,T];\,FL^p)}2^{k(s+\frac2{r_1})}\|\dot{\Delta}_k u\|_{L^{r_1}([0,T];\,FL^p)}.
\end{align*}
Then, we choose $r_1\in (\frac{r}{r-1},\infty)$ such that $s+\frac1r+\frac1{r_1}>0$, that is $\frac1{r_1}\in (\frac3p-2-\frac1r,1-\frac 1r)$, we get that
\[\|tIII\|_{\widetilde{L}^{\infty}([0,T];\,F\dot{B}^{s+2}_{p,q})}
\lesssim_{\nu}T ^{\frac{s-s_p}2}\|t^{\frac1r}v\|_{\widetilde{L}^{\infty}([0,T ];\,F\dot{B}^{s+\frac2r}_{p,q})}
\|u\|_{\widetilde{L}^{r_1}([0,T];\,F\dot{B}^{s+\frac2{r_1}}_{p,q})}.\]
By the interpolation and Young inequalities, we have
\begin{equation}\label{eq.ana3}
\begin{split}
\|tIII\|_{\widetilde{L}^{\infty}([0,T ];\,F\dot{B}^{s+2}_{p,q})}
\lesssim&_{\nu}T ^{\frac{s-s_p}2}\|tv\|^{1/r}_{\widetilde{L}^{\infty}([0,T];\,F\dot{B}^{s+\frac2r}_{p,q})}
\|v\|^{1-1/r}_{\widetilde{L}^{\infty}([0,T];\,F\dot{B}^{s}_{p,q})}
\|u\|_{\widetilde{L}^{r_1}([0,T];\,F\dot{B}^{s+\frac2{r_1}}_{p,q})}.\\
\lesssim&_{\nu}T ^{\frac{s-s_p}2}\|tv\|_{\widetilde{L}^{\infty}([0,T];\,F\dot{B}^{s+\frac2r}_{p,q})}
\|u\|_{\widetilde{L}^{r_1}([0,T];\,F\dot{B}^{s+\frac2{r_1}}_{p,q})}\\
&+T ^{\frac{s-s_p}2}\|v\|_{\widetilde{L}^{\infty}([0,T];\,F\dot{B}^{s}_{p,q})}
\|u\|_{\widetilde{L}^{r_1}([0,T];\,F\dot{B}^{s+\frac2{r_1}}_{p,q})}.
\end{split}
\end{equation}
Collecting estimates \eqref{eq.ana1}-\eqref{eq.ana3}, we immediately obtain
\begin{equation}\label{eq.ana4}
\|t B(u,v)\|_{\widetilde{L}^{\infty}([0,T ];\,F\dot{B}^{s+2}_{p,q})}\lesssim_{\nu} T^{\frac{s-s_p}2}\|u\|_{Y_{T }}
\|tv\|_{\widetilde{L}^{\infty}([0,T];\,F\dot{B}^{s+2}_{p,q})}
+T^{\frac{s-s_p}2}\|u\|_{Y_{T }}\|v\|_{Y_{T }}.
\end{equation}
In the same way,  we can show
\begin{equation}\label{eq.ana5}
\|tB(v,u)\|_{\widetilde{L}^{\infty}([0,T];\,F\dot{B}^{s+2}_{p,q})}\lesssim_{\nu} T^{\frac{s-s_p}2}\|u\|_{Y_{T }}
\|tv\|_{\widetilde{L}^{\infty}([0,T];\,F\dot{B}^{s+2}_{p,q})}+T^{\frac{s-s_p}2}\|u\|_{Y_{T }}\|v\|_{Y_{T }}.
\end{equation}
Finally, we tackle with the  term $e^{\nu t\Delta}u_0$. Note that
\begin{align*}
2^{j(s+2)}t\|\dot{\Delta}_je^{\nu t\Delta} u_0\|_{FL^p}\lesssim 2^{j(s+2)}te^{-c\nu t2^{2j}}\|\dot{\Delta}_ju_0\|_{FL^p}\lesssim_{\nu}2^{js}\|\dot{\Delta}_ju_0\|_{FL^p},
\end{align*}
we have
\begin{equation}\label{eq.ana6}
\|te^{\nu t\Delta} u_0\|_{\widetilde{L}^{\infty}([0,T];\,F\dot{B}^{s+2}_{p,q})}\lesssim_{\nu}\|u_0\|_{F\dot{B}^{s}_{p,q}}.
\end{equation}
Combining estimates \eqref{eq.ana4}-\eqref{eq.ana6}, by Lemma \ref{regul} and continuity argument, we finally get
$$\|t u(t)\|_{\widetilde{L}^{\infty}([0,T];\,F\dot{B}^{s+2}_{p,q})}<\infty.$$
Thus we complete the proof of Proposition \ref{pro.ana}.
\end{proof}
\begin{remark}\label{rem.ana}
 From Proposition \ref{pro.ana}, we see that for any $q<\infty$
\[\lim_{T\to 0+}\|tu(t)\|_{\widetilde{L}^{\infty}([0,T];\,F\dot{B}^{s+2}_{p,q})}=0.\]
 For any $\varepsilon>0$, from \eqref{eq.ana6} and the fact that $q<\infty$, we can choose $N\in \NN$ such that
\begin{equation}\label{eq.G-1}
\sum_{|j|\geq N}2^{j(s_p+2)q}\|te^{\nu t\Delta} u_0\|^q_{L^{\infty}([0,T];\,FL^p)}<\varepsilon^q/2^q.
\end{equation}
Since $te^{\nu t\Delta}$ is tends to zero  as $t\to 0+$, there exists a $\delta>0$ such that for any $t\in[0,\delta)$
\[\|te^{\nu t\Delta} u_0\|_{FL^p}\leq 2^{-1+(2N-1)(s_p+2)}\varepsilon,\]
which implies that for any $0\leq T<\delta$
\[\|te^{\nu t\Delta} u_0\|^q_{L^{\infty}([0,T];\,FL^p)}<2^{-1+(2N-1)(s_p+2)}\varepsilon.\]
Thus, for any $0\leq T<\delta$
\begin{equation}\label{eq.G-2}
\sum_{|j|<N}2^{j(s_p+2)q}\|te^{\nu t\Delta} u_0\|^q_{L^{\infty}([0,T];\,FL^p)}<\varepsilon^q/2^q.
\end{equation}
Combining \eqref{eq.G-1} and \eqref{eq.G-2} yields
\[\lim_{T\to 0+}\|te^{\nu t\Delta} u_0\|_{\widetilde{L}^{\infty}([0,T];\,F\dot{B}^{s+2}_{p,q})}=0.\]
With this property, we have by Proposition \ref{pro.ana}   that
\[\lim_{T\to 0+}\|tu\|_{\widetilde{L}^{\infty}([0,T];\,F\dot{B}^{s+2}_{p,q})}=0.\]
  Moreover, we readily obtain that
\begin{equation}\label{eq.ana-2}
\sup_{0\leq t< T }t^{\frac \alpha 2-\frac s2+\frac{3}{2\eta}-\frac 3{2p}}\|\nabla^{\alpha}
u(t)\|_{FL^{\eta}}<\infty,\quad \forall (\alpha,\eta)\in \{0,1\}\times [1,p]
\end{equation}
as $s\in (s_p,0)$,
and
\begin{equation}\label{eq.ana-3}
  \sup_{0\leq t< T }t^{\frac \alpha 2-\frac s2+\frac{3}{2\eta}-\frac 3{2p}}\|\nabla^{\alpha}
   u(t)\|_{FL^{\eta}}<\infty,\quad \forall (\alpha,\eta)\in\big\{ \{0,1\}\times [1,p]\big\}\backslash(1,1)
 \end{equation}  as $s=s_p$.

Indeed, by Minkowski's inequality, we have
\[\sup_{0\leq t\leq T}t\|u(t)\|_{F\dot{B}^{s+2}_{p,q}}<\infty\]
Then, by the sharp interpolation inequality, we have that for any $0\leq t\leq T$
\begin{align*}
\|\nabla^{\alpha}u(t)\|_{FL^{\eta}}\leq \|u(t)\|_{F\dot{B}^{\alpha}_{\eta,1}} \leq&\|u(t)\|_{F\dot{B}^{s+2+\frac3p-\frac3\eta}_{\eta,q}}^{\frac\alpha 2+\frac{3}{2\eta}-\frac{3}{2p}-\frac s2}
\|u(t)\|^{1-\frac\alpha 2-\frac{3}{2\eta}+\frac{3}{2p}+\frac s2}_{F\dot{B}^{s+\frac3p-\frac3\eta}_{\eta,q}}\\
\leq&\|u(t)\|_{F\dot{B}^{s+2}_{p,q}}^{\frac{3}{2\eta}-\frac{3}{2p}-\frac s2}
\|u(t)\|^{1-\frac{3}{2\eta}+\frac{3}{2p}+\frac s2}_{F\dot{B}^{s}_{p,q}}
\end{align*}
for any $(\alpha,\eta)\in \{0,1\}\times [1,p]$ if $s\in (s_p,0)$ or
$\big\{(\alpha,\eta)\in \{0,1\}\big\}\times [1,p]\setminus(1,1)$ if $s=s_p$.

 This gives that
\[\sup_{0\leq t\leq T}t^{\frac\alpha 2-\frac s2+\frac{3}{2\eta}-\frac{3}{2p}}\|u(t)\|_{FL^p}\leq \Big(\sup_{0\leq t\leq T}t\|u(t)\|_{F\dot{B}^{s+2}_{p,q}}\Big)^{\frac\alpha 2+\frac{3}{2\eta}-\frac{3}{2p}-\frac s2}
\|u(t)\|^{1-\frac\alpha 2-\frac{3}{2\eta}+\frac{3}{2p}+\frac s2}_{F\dot{B}^{s}_{p,q}}.\]
\end{remark}

\section{Some properties of weak energy solutions to perturbed equations}\label{sec.4}
\setcounter{section}{4}\setcounter{equation}{0}
In this section, we will study some properties of weak energy solutions to the perturbed equations
 \begin{equation}\label{W-2}
\left\{\begin{array}{ll}
\partial_t w-\nu\Delta w+w\cdot\nabla w+v\cdot\nabla w+w\cdot\nabla v+\nabla \bar{P}=0,\\
\Div w=0,\\
w(x,0)=w_0(x)\in L^2(\RR^3).
\end{array}\right.
\end{equation}
with that $v$ is a mild solution to \eqref{NS} established in Theorem \ref{well-sub}, satisfying $$v\in C([0,T];\,F\dot{B}^{s}_{\tilde{p},\tilde{q}}(\RR^3)) \cap \widetilde{L}^r([0,T];\,F\dot{B}^{2+s}_{\tilde{p},\tilde{q}}(\RR^3))\,\,\,\quad\forall r\in [1,\infty]$$
with $\tilde{p}\in (1,3/2)$ and $s\in (s_{\tilde{p}},0)$.
Since  $w_0\in L^2(\RR^3)$ satisfies $\Div w_0=0$, by the Fardo-Galerkin method used in  \cite{Al17}, we can show that problem \eqref{W-2}  admits at most one  weak energy solution $(w,\bar{P})$ satisfies the following conditions:
\begin{enumerate}
  \item [\rm (W1)] \label{W1} $w\in L^{\infty}((0,T );\,L^2(\RR^3))\cap L^2((0,T );\,\dot{H}^1(\RR^3)), $ $\bar{P}\in \left(L^{{3/2}}+L^{2}\right)(\RR^3\times (0,T)).$\medskip

  \item[\rm (W2)] \label{W2}$(w,\,\bar{P})$ satisfies Eq. \eqref{W-2} in the sense of distribution on $\RR^3\times (0,T )$.\medskip

  \item [\rm (W3)]\label{W3}$\displaystyle\lim_{t\to 0+}\|w(t)-w_0\|_{L^2}=0$.\medskip

  \item[\rm (W4)] \label{W4} For any $\varphi\in \mathcal{D}(\RR^3\times (0,T))$, $\varphi\geq 0$, the generalized local energy inequality holds
  \begin{equation}\label{loc.W}
  \begin{split}
   &\int_{\RR^3\times \{t\}}|w|^2\varphi\,\mathrm{d}x+2\nu\int^t_0\int_{\RR^3}|\nabla w|^2\varphi\,\mathrm{d}x\mathrm{d}s\\
  \leq&\int^t_0\int_{\RR^3}|w|^2(\partial_t\varphi+\nu\Delta\varphi)\,\mathrm{d}x\mathrm{d}s+\int^t_0\int_{\RR^3}|w|^2(w+v)\cdot\nabla
   \varphi\,\mathrm{d}x\mathrm{d}s\\
   &+2\int^t_0\int_{\RR^3}(w\cdot\nabla \varphi)(v\cdot w)+(w\cdot\nabla w)\cdot v\varphi\,\mathrm{d}x\mathrm{d}s+2\int^t_0\int_{\RR^3}\bar{P} w\cdot\nabla \varphi\,\mathrm{d}x\mathrm{d}s.\\
        \end{split}
       \end{equation}
       \item[\rm (W5)] \label{W5} $w$ satisfies the generalized global energy inequality on $\RR^3\times (0,T)$:
        \begin{equation}\label{glo.W}
        \int_{\RR^3\times\{t\}}|w|^2\,\mathrm{d}x+2\nu\int^t_0\int_{\RR^3}|\nabla w|^2\,\mathrm{d}x\mathrm{d}s\leq \|w_0\|^2_{L^2}+2\int^t_0(w\cdot\nabla) w\cdot v\,\mathrm{d}x\mathrm{d}s.
        \end{equation}
\end{enumerate}
Following the argument used in \cite{CKN82}, we can show that this weak energy solution is also the suitable weak solution which is defined as follows:
\begin{definition}[Suitable weak solutions]\label{def.suit}
Let $D=\Omega\times (0,T)$ be an open set of $\RR^3\times\RR^+$ and $v\in L^p_tL^q_x(D)$ satisfying $\Div v=0$, ${2\over p}+{3\over q}\leq 1$ and $3\leq q\leq\infty$. Then we call $(w,\bar{P})$ is a suitable weak solution in $D$ to equations:
\begin{equation}\label{eq.W'}
\left\{\begin{array}{ll}
\partial_t w-\nu\Delta w+w\cdot\nabla w+v\cdot\nabla w+w\cdot\nabla v+\nabla \bar{P}=0\\
\Div w=0\\
\end{array}\right.,
\end{equation}
if $(w,\bar{P})$ satisfies
\begin{enumerate}
\item[$(\mathrm{i})$] $w\in L^{\infty}L^2(D)\cap L^2\dot{H}^1(D)$ and $Q\in L^{3/ 2}(D)+L^2(D);$\smallskip

\item[$(\mathrm{ii})$] $w$ and $\bar{P}$ satisfy Eq. \eqref{eq.W'} in the sense of distribution on $D;$\smallskip

\item[$(\mathrm{iii})$] $w$ and $\bar{P}$ satisfy the generalized local energy inequality \eqref{loc.W} on $D$.
\end{enumerate}
\end{definition}

In the following part of Section \ref{sec.4}, we will discuss the partial regularity and stability of singularities of suitable weak solutions to problem \eqref{eq.W'}, and the short-time behaviour of the kinetic energy  and weak-strong uniqueness of the local energy solutions to problem \eqref{eq.W'}.
\subsection{Partial regularity criterion}
In this subsection, inspired by  CKN's theorem established in  \cite{CKN82}, we will shown the following $\varepsilon$-regularity criteria of  suitable weak solutions to the  perturbed equations \eqref{eq.W'}, by  De Giorgi iteration and dimensional analysis.
\begin{theorem}[$\varepsilon$-regularity criterion]\label{regularity}
 Let  $v\in L^p_tL^q_x(Q_R(x_0,t_0))$ satisfies $\Div v=0$ with  ${2\over p}+{3\over q}<1$, $q> 3$. Assume $(w,\bar{P})$ is a suitable weak solution to Eq. \eqref{eq.W'} in $Q_R(x_0,t_0)$. Then there exists an absolute constant $\epsilon_0=\epsilon_0(p,q,\|v\|_{L^p_tL_x^q})>0$ with the following property: if
\[R^{-2}\int_{Q_R(x_0,t_0)}\left(|w|^3+|\bar{P}|^{3\over 2}\right)\,\mathrm{d}x\mathrm{d}\tau< \epsilon_0,\]
then there exists a constant $C^*>0$ such that
$$\|w\|_{L^{\infty}(Q_{R/2}(x_0,t_0))}\leq C^*R^{-1}.$$
\end{theorem}
\begin{remark}
Let us point out that in  \cite{JS14},  Jia and \v{S}ver\'{a}k  proved the same  result under the condition  $v\in L^{m}(Q_R(x_0,t_0))$ with $m>5$. Here, we  consider the general case that $v\in L^p_tL^q_x(Q_R(x_0,t_0))$ and   give a new proof is based on  De Giorgi  iteration and  dimensional analysis.
\end{remark}
Before proving Theorem \ref{regularity}, we first give two  useful lemmas.
\begin{lemma}[\cite{Ev86}]\label{lem.uni}
Let $f$ be a nonnegative nondecreasing bounded function defined on $[0,1]$ with the following property:
for any $0\leq {r}\leq s<t<R\leq 1$ and some constants $\theta\in(0,1)$, $M>0$, $\beta>0$, we have
\[f(s)\leq \theta f(t)+{M\over (t-s)^{\beta}}.\]
Then,
\[\sup_{s\in [0,r]}f(s)\leq C(\theta,\beta){M\over (R-r)^{\beta}}.\]
\end{lemma}
\begin{lemma}\label{est.inter}
There exists a constant $C>0$ such that for any $u\in L^{\infty}_tL^2_x(Q_r)$ and $\nabla u\in L^2_tL^2_x(Q_r)$,
\[\|u\|_{L^m_tL^n_x(Q_r)}\leq C(\|u\|_{L^{\infty}_tL^2_x(Q_r)}+\|\nabla u\|_{L^2(Q_r)})\]
for any $2/m+3/n=3/2$ with $2\leq n\leq 6$.
\end{lemma}
\begin{proof}
By Gagliardo-Nirenberg's inequality, we readily have
\[\|u(t)\|_{L^6_x(B_r)}\leq C\|\nabla u(t)\|_{L^2_x(B_r)}+Cr^{-1}\|u(t)\|_{L^2_x(B_r)}.\]
Moreover, we have by H\"{o}lder's inequality that
\[\|u\|_{L^2_tL^6_x(B_r)}\leq C(\|\nabla u\|_{L^2(Q_r)}+\|u\|_{L^{\infty}_tL^2_x(B_r)}).\]
Hence, the  interpolation between $L^{\infty}_tL^2_x(Q_r)$ and $L^2_tL^6_x(Q_r)$ entails
\[\|u\|_{L^m_tL^n_x(Q_r)}\leq C\big(\|u\|_{L^{\infty}_tL^2_x(Q_r)}+\|\nabla u\|_{L^2(Q_r)}\big),\]
where ${2/ m}+{3/n}=2/3$ and $2\leq n\leq 6$.
\end{proof}
Now, we come back to the proof of Theorem \ref{regularity}.
Letting $R_0<{R\over 2}$, we find that $$Q_{R_0}(x_1,t_1)\subset Q_R(x_0,t_0) \quad \text{for each}\quad(x_1,t_1)\in \overline {Q_{\frac R 2}(x_0,t_0)}.$$
Denoting
\[w_{R_0}(x,t)=R_0 w\left(x_1+R_0 x,t_1+R_0^2 t\right),\quad\bar{P}_{R_0}(x,t)=R_0^2 \bar{P}\left(x_1+R_0 x,t_1+R_0^2 t\right),\]
\[v_{R_0}(x,t)=R_0 v\left(x_1+R_0 x,t_1+R_0^2 t\right),\]
we find that the couple $(w_{R_0},\bar{P}_{R_0})$ is also a suitable weak solution to equations \eqref{eq.W'} in $Q_{1}$ with $v_{R_0}$ instead of $v$. And we have
\[\int_{Q_{1}}\left(|w_{R_0}|^3+|\bar{P}_{R_0}|^{3/2}\right)\,\mathrm{d}x\mathrm{d}\tau=R_0^{-2}\int_{Q_{R_0}(x_1,t_1)}\left(|w|^3+|\bar{P}|^{3/2}\right)\,\mathrm{d}x\mathrm{d}\tau
\leq R^{-2}_0\epsilon_0\]
and
\[\|v_{R_0}\|_{L^p_tL^q_x(Q_1)}=R_0^{1-2/p-3/q}\|v\|_{L^p_tL^q_x(Q_{R_0}(x_1,t_1))}\leq R_0^{1-2/p-3/q}\|v\|_{L^p_tL^q_x(Q_{R}(x_0,t_0))}.\]
Set $\epsilon_1=R_0^{1-{2/p}-{3/q}}\|v\|_{L^p_tL^q_x(Q_{R}(x_0,t_0))}$. Since $2/p+3/q<1$, we can choose $R_0$ small enough such that $\epsilon_1\ll{1/2}$. Fixed $R_0$, then choose suitable small $\epsilon_0$ such that $R^{-2}_0\epsilon_0<\epsilon_1$. In the following proof, without ambiguity, we still use $(w,\bar{P},v)$ to denote $(w_{R_0},\bar{P}_{R_0},v_{R_0})$. Therefore, we have
\begin{equation*}\label{con.sma}
\int_{Q_{1}}\left(|w|^3+|\bar{P}|^{3/2}\right)\,\mathrm{d}x\mathrm{d}\tau+\|v\|_{L^p_tL^q_x(Q_1)}< 2\epsilon_1.
\end{equation*}
This estimate allows us to  claim that there exists a positive constant $C^*$ such that
\begin{equation}\label{w-bound}
\|w\|_{L^{\infty}(Q_{\frac1 2})}\leq C^*.
\end{equation}
By taking the inverse transform of scaling,  we immediately obtain the theorem.

 We turn to prove claim \eqref{w-bound}.   For any $k\in \{0,1,2,\cdots\}$, we introduce a new function
\[w_k=\left[|w|-\left(1-2^{-k}\right)\right]_+.\]
Since $w_k^2$ equals to $0$ for $|w|<1-2^{-k}$ and is of the order of $|w|^2$ for $|w|\gg1-2^{-k}$, $w^2_k$ can be seen as a level set of energy.

Let us denote
\[r_k=2^{-1}+2^{-k-2},\quad r_{k-1/3}=2^{-1}+2^{-k-2+1/3},\quad r_{k-2/3}=2^{-1}+2^{-k-2+2/3}.\]
Then, we define
\[W_k\triangleq\mbox{ess}\sup_{-r^2_k<t\leq 0}\int_{B_{r_k}}|w_k|^2\,\mathrm{d}x+\int_{Q_{r_k}}|d_k|^2\,\mathrm{d}x\mathrm{d}\tau\]
where
\[d^2_k\triangleq 2\nu{(1-2^{-k})I_{\{|w|\geq 1-2^{-k}\}}\over |w|}|\nabla|w||^2+2\nu{w_k\over |w|}|\nabla w|^2.\]

For such $w_k$, $d_k$ and $W_k$, we have the following well-known properties which will be useful in the following proof.
\begin{lemma}[\cite{Va07}]\label{lem.link}
In the light of the definition of $w_k$ and $d_k$, the function $w$ can be decomposed as follows:
\[w=w\bigg(1-{w_k\over |w|}\bigg)+w{w_k\over |w|}\]
satisfying
\[\bigg|w\bigg(1-{w_k\over |w|}\bigg)\bigg|\leq 1-2^{-k}.\]
Moreover, we can control the following gradients by $d_k$:
\[{w_k\over |w|}|\nabla w|\leq \nu^{-\frac12}d_k,\quad I_{\{|w|\geq 1-2^{-k}\}}\big|\nabla |w|\big|\leq \nu^{-\frac12}d_k,\]
\[|\nabla w_k|\leq \nu^{-\frac12}d_k\quad\text{and}\quad \Big|\nabla {ww_k\over |w|}\Big|\leq 3\nu^{-\frac12}d_k.\]
\end{lemma}
\begin{lemma}\label{lem.est}
For any $2\leq m\leq \infty$, $2\leq n\leq 6$ satisfying ${2\over m}+{3\over n}={3\over 2}$, there exists a constant $C(\nu)>0$ such that, for any $k\geq 0$,
\[\|w_k\|_{L^m_tL^n_x(Q_{r_k})}\leq C(\nu) W^{{1/ 2}}_{k}.\]
\end{lemma}
\begin{proof}
According to Lemma \ref{lem.uni} and Lemma \ref{lem.link} we have
\begin{align*}
\|w_k\|_{L^m_tL^n_x(Q_{r_k})}\lesssim& \|w_k\|_{L^{\infty}_tL^2(Q_{r_k})}+\|\nabla w_k\|_{L^2(Q_{r_k})}\\
\lesssim_{\nu}&\|w_k\|_{L^{\infty}_tL^2_x(Q_{r_k})}+\|d_k\|_{L^2(Q_{r_k})}\lesssim_{\nu} W^{1/2}_k.
\end{align*}
This implies the desired estimate.
\end{proof}
\begin{lemma}\label{lem.w_k}
In the light of the definition of $w_k$ and $W_k$, there exists a constant $C(\nu)>0$ such that, for any $k\geq 1$, $q>1$ and ${2\over m}+{3\over n}=1$ with $2\leq m\leq \infty$, we have
\[\|I_{w_k>0}\|_{L^q(Q_{k-1})}\leq C(\nu)2^{{10k/3q}}W_{k-1}^{{5/3q}},\quad \|I_{w_k>0}\|_{L^{\infty}_tL^q_x(Q_{r_{k-1}})}\leq C(\nu)2^{{2k/q}}W_{k-1}^{{1/q}}\]
and
\[\|I_{w_k>0}\|_{L^{m}_tL^n_x(Q_{r_{k-1}})}\leq C(\nu)2^{{k}}W_{k-1}^{{1/2}}.\]
\end{lemma}
\begin{proof}
	Here we omit the proof, because it is standard and is similar with that of \cite{Va07}.
\end{proof}
Now we  begin to prove \eqref{w-bound} step by step.\smallskip

\noindent\textbf{Step 1.} $W_0\leq C_1(\nu)\varepsilon_1$ for some $C_1(\nu)>0$.

For any $\rho>0$, set
$$E(\rho)=\mathop{\mbox{ess}\sup}\limits_{-\rho^2<t\leq 0}\int_{B_{\rho}}|w|^2\,\mathrm{d}x+2\nu\int_{Q_{\rho}}|\nabla w|^2\,\mathrm{d}x\mathrm{d}\tau.$$
For any ${3/4}\leq \rho_1<\rho_2\leq 1$, let $\varphi\in \mathcal{D}(Q_{\rho_2})$ satisfying
\[0\leq \varphi\leq 1,\quad \varphi\equiv 1\,\quad\text{ in }\,\,Q_{\rho_1}\]
and \[(\rho_2-\rho_1)^2|\partial_t \varphi|+\sum^2_{i=1}(\rho_2-\rho_1)^i|\nabla^i\varphi|\leq C\,\,\,\text{ in }\,\,\,Q_{\rho_2}.\]
Applying $\varphi$ to the local energy inequality \eqref{loc.W}, by H\"{o}lder's inequality and Lemma \ref{lem.uni}, we have
\begin{align*}
E(\rho_1)\leq& \frac {C(\nu)}{(\rho_2-\rho_1)^2} \int_{Q_{\rho_2}}|w|^2\,\mathrm{d}x\mathrm{d}\tau+\frac C{\rho_2-\rho_1}\int_{Q_{\rho_2}}\big(|w|^3+|w|^2|v|\big)\,\mathrm{d}x\mathrm{d}\tau\\
&+C\int_{Q_{\rho_2}}|w||\nabla w||v|\,\mathrm{d}x\mathrm{d}\tau+\frac C{\rho_2-\rho_1}\int_{Q_{\rho_2}}|\bar{P}||w|\,\mathrm{d}x\mathrm{d}\tau\\
\leq &\frac {C(\nu)}{(\rho_2-\rho_1)^2}\|w\|^{2}_{L^{3}(Q_1)}
+\frac C{\rho_2-\rho_1}\|w\|^{3}_{L^{3}(Q_1)}+\frac C{\rho_2-\rho_1}\|w\|^2_{L^{4q/3}_tL^{2q/(q-1)}_x(Q_1)}\|v\|_{L^p_tL^q_x(Q_1)}\\
&+C\|w\|_{L^{2q/3}_tL^{2q/(q-2)}_x(Q_{\rho_2})}\|\nabla w\|_{L^2(Q_{\rho_2})}\|v\|_{L^p_tL^q_x(Q_1)}+\frac C{\rho_2-\rho_1}\|w\|_{L^{3}(Q_1)}\|\bar{P}\|_{L^{{3/2}}(Q_1)}\\
\leq &\frac{C(\nu)\epsilon^{2/3}_1}{(\rho_2-\rho_1)^2}+C(\nu)\epsilon_1E(\rho_2).
\end{align*}
Choosing $\epsilon_1$ small enough such that $C(\nu)\epsilon_1<1$,  we get by Lemma \ref{lem.uni} that
\begin{equation}\label{est.0}
W_0=E(3/4)\leq C_1(\nu)\epsilon^{2/3}_1.
\end{equation}
\textbf{Step 2. } $W_k\leq C^k_2(\nu)W^{\beta}_{k-1}$ with some $C_2(\nu)>1,\,\beta>1$ for any $k\geq 1$.
\medskip

Since $|w({w_k\over |w|}-1)|\leq 1$, we can multiply the first equation of \eqref{eq.W'} by $w({w_k\over |w|}-1)$, we immediately have
\begin{align*}
&\partial_t{w_k^2-|w|^2\over 2}-\nu\Delta{w^2_k-|w|^2\over 2}-|\nabla w|^2+d^2_k+\Div \bigg((w+v){w^2_k-|w|^2\over 2}\bigg)\\
&+w\bigg({w_k\over |w|}-1\bigg)\Div(v\otimes w)+\nabla \bar{P}\cdot w\bigg({w_k\over |w|}-1\bigg)=0
\end{align*}
in the sense of distribution on $Q_1$.

This  together with \eqref{loc.W}  gives that
\begin{equation}\label{loc.w_k}
\begin{split}
&\int_{\RR^3\times \{t\}}|w_k|^2\varphi\,\mathrm{d}x+2\int^t_0\int_{\RR^3}d_k^2\varphi\,\mathrm{d}x\mathrm{d}\tau\\
 \leq&\int^t_0\int_{\RR^3}|w_k|^2(\partial_t\varphi+\nu\Delta\varphi)\,\mathrm{d}x\mathrm{d}\tau
+\int^t_0\int_{\RR^3}|w_k|^2(w+v)\cdot\nabla\varphi\,\mathrm{d}x\mathrm{d}\tau\\
&+2\int^t_0\int_{\RR^3}(w\cdot\nabla \varphi)\bigg(v\cdot w{w_k\over |w|}\bigg)+\bigg(w\cdot\nabla w{w_k\over |w|}\bigg)\cdot v\varphi\,\mathrm{d}x\mathrm{d}\tau\\
&+2\int^t_0\int_{\RR^3}\bar{P}w\cdot\nabla \varphi\,\mathrm{d}x\mathrm{d}\tau-2\int^t_0\int_{\RR^3}\nabla \bar{P}\cdot w\bigg({w_k\over |w|}-1\bigg)\varphi\,\mathrm{d}x\mathrm{d}\tau
\end{split}
\end{equation}
for any nonnegative function $\varphi\in \mathcal{D}(Q_1)$.

For any $k\in \NN$, let $\varphi_k\in \mathcal{D}(Q_1)$ satisfying
\[0\leq\varphi_k\leq 1,\quad\varphi_k\equiv 1\,\quad\text{ in }\,\,\, Q_{r_k},\quad \varphi_k\equiv 0\quad\text{ in }\,\, Q^c_{r_{k-1/3}},\]
\[2^{2k}\|\partial_t \varphi_k\|_{L^{\infty}(Q_1)}+\sum^2_{i=1}2^{ik}\|\nabla^i \varphi_k\|_{L^{\infty}(Q_1)}\leq C.\]
Applying $\varphi_k$ to \eqref{loc.w_k}, by Lemma \ref{lem.link}, H\"{o}lder's inequality and the fact ${2/p}+{3/q}<1$, we can get that
\begin{align*}
W_k\lesssim &_{\nu} 2^{2k}\int_{Q_{r_{k-1}}}|w_k|^2\,\mathrm{d}x\mathrm{d}\tau+2^k\int_{Q_{r_{k-1}}}|w_k|^3\,\mathrm{d}x\mathrm{d}\tau
+2^k\int_{Q_{r_{k-1}}}\left(|w_k|^2|v|+|v||w_k|\right)\,\mathrm{d}x\mathrm{d}\tau\\
& +\int_{Q_{r_{k-1}}}\left(I_{\{w_k>0\}}|d_k||v|+|w_k||d_k||v|\right)\,\mathrm{d}x\mathrm{d}\tau\\
&+\bigg|\int^t_0\int_{\RR^3}\bigg(\bar{P}w\cdot\nabla \varphi_k-\nabla\bar{P}\cdot w\bigg({w_k\over |w|}-1\bigg)\varphi_k\bigg)\,\mathrm{d}x\mathrm{d}\tau\bigg|.
\end{align*}
We now denote by $\text{I.1}$ and $\text{I.2}$ the first four terms and the last term on the right-hand side of the above inequality, respectively.

\textbf{Estimate of $\text{I.1}$.}
Thanks to H\"{o}lder's inequality and the fact that $2/p+3/q<1$, we have the following estimate:
\begin{align*}
\text{I.1}\lesssim &_{\nu} 2^{2k}\|w_k\|^2_{L^{10/3}(Q_{r_{k-1}})}\|I_{\{w_k>0\}}\|_{L^{5/2}(Q_{r_{k-1}})}
+2^k\|w_k\|^3_{L^{10/3}(Q_{r_{k-1}})}\|I_{\{w_k>0\}}\|_{L^{10}(Q_{r_{k-1}})}\\
&+2^k\|w_{k}\|^2_{L^{2p/(p-1)}_tL^{6p/(p+2)}_x(Q_{r_{k-1}})}\|v\|_{L^p_tL^q_x(Q_{r_{k-1}})}
\|I_{\{w_k>0\}}\|_{L^{\infty}_tL^{l_1}_x(Q_{r_{k-1}})}\\
&+2^k\|w_{k}\|_{L^{2p/(p-1)}_tL^{6p/(p+2)}_x(Q_{r_{k-1}})}\|I_{\{w_{k}>0\}}\|_{L^{2p/(p-1)}_tL^{l_2}_x(Q_{r_{k-1}})}
\|v\|_{L^p_tL^q_x(Q_{r_{k-1}})}\\
&+\|d_k\|_{L^2(Q_{r_{k-1}})}\|v\|_{L^p_tL^q_x(Q_{r_{k-1}})}
\|I_{\{w_k>0\}}\|_{L^{2p/(p-2)}_tL^{2q/(q-2)}_x(Q_{r_{k-1}})}\\
&+\|d_k\|_{L^2(Q_{r_{k-1}})}
\|w_k\|_{L^{2p/(p-2)}_tL^{6p/(p+4)}_x(Q_{r_(k-1)})}\|v\|_{L^p_tL^q_x(Q_{r_{k-1}})}
\|I_{\{w_k>0\}}\|_{L^{\infty}_tL^{l_3}_x(Q_{r_{k-1}})},
\end{align*}
where ${1\over l_1}={2\over 3}-{2\over 3p}-{1\over q}>{1\over 3}$, ${1\over l_2}={1\over l_1}+{p+2\over 6p}$ and ${1\over l_3}={1\over 3}-{2\over 3p}-{1\over q}>0$.

By Lemma \ref{lem.est}, Lemma \ref{lem.w_k}, H\"{o}lder's inequality and the fact $w_k\leq w_{k-1}$, we  deduce that
\begin{equation}\label{est.1}
\begin{split}
\text{I.1}\lesssim&_{\nu} 2^{{10k\over 3}}W^{{5\over 3}}_{k-1}+2^{{4k\over 3}}W^{{5\over 3}}+\Big(2^{({7\over 3}-{4\over 3p}-{2\over q})k}W^{{5\over 3}-{2\over 3p}-{1\over q}}_{k-1}+2^{({10\over 3}-{4\over 3p}-{2\over q})k}W^{{5\over 3}-{2\over 3p}-{1\over q}}_{k-1}\\
&+2^{({5\over 3}-{4\over 3p}-{2\over q})k}W^{{4\over 3}-{2\over 3p}-{1\over q}}_{k-1}+2^{({2\over 3}-{4\over 3p}-{2\over q})k}W^{{4\over 3}-{2\over 3p}-{1\over q}}_{k-1}\Big)\|v\|_{L^p_tL^q_x(Q_1)}.
\end{split}
\end{equation}

\textbf{Estimate of $\text{I.2}$. } Denote $w$ and $v$ as $w=(w^1,w^2,w^3)$ and $v=(v^1,v^2,v^3)$. From \eqref{eq.W'}, we see that
\[-\Delta \bar{P}=\partial_i\partial_j(w^iw^j+w^iv^j+v^iw^j) \quad \text{in}\quad Q_1.\]
Letting $\phi_k\in \mathcal{D}(B_{k-1})$ with $0\leq \phi\leq 1$ and $\phi\equiv 1$ in $B_{k-{2/3}}$, we have for any $x\in B_{k-2/3}$
\begin{align*}
\bar{P}(x,t)=&\phi_k(x) \bar{P}(x,t)\\
=&-{3\over 4\pi}\int_{\RR^3}{1\over |x-y|}\left(\phi_k(y)\Delta \bar{P}(y,t)+2\nabla \phi_k(y)\cdot\nabla \bar{P}(y,t)+\bar{P}(y,t)\Delta \phi_k(y)\right)\,\mathrm{d}y\\
\triangleq& \bar{P}_{k1}(x,t)+\bar{P}_{k2}(x,t)+\bar{P}_{k3}(x,t),
\end{align*}
where
\[\bar{P}_{k1}(x,t)={3\over 4\pi}\int_{\RR^3}\partial_i\partial_j\Big({1\over |x-y|}\Big)\phi_k(y)(w^iw^j)(y,t)\,\mathrm{d}y,\]
\[\bar{P}_{k2}(x,t)={3\over 4\pi}\int_{\RR^3}\partial_i\partial_j\Big({1\over |x-y|}\Big)\phi_k(y)(w^iv^j+v^iw^j)(y,t)\,\mathrm{d}y,\]
and
\begin{align*}
\bar{P}_{k3}(x,t)=&{3\over 2\pi}\int_{\RR^3}{x_i-y_i\over |x-y|^3}(\partial_j\phi_k(y))(w^iw^j+w^iv^j+v^iw^j)(t,y)\,\mathrm{d}y\\
&+{3\over 4\pi}\int_{\RR^3}{1\over |x-y|}(\partial_i\partial_j\phi_k(y))(w^iw^j+w^iv^j+v^iw^j)(y,t)\,\mathrm{d}y\\
&+{3\over 4\pi}\int_{\RR^3}{1\over |x-y|}\bar{P}(y)\Delta\phi_k(y)\,\mathrm{d}y+{3\over 2\pi}\int_{\RR^3}{x_i-y_i\over |x-y|^3}\bar{P}(y)\partial_i\phi_k(y)\,\mathrm{d}y.
\end{align*}
So, $\text{I.2}$ can be estimated as follows:
\begin{align*}
\text{I.2}\leq& \bigg|\int_{Q_{r_{k-1/3}}}\bigg(\bar{P}_{k1} w\cdot\nabla \varphi_k-\nabla \bar{P}_{k1}\cdot w\bigg({w_k\over |w|}-1\bigg)\varphi_k\bigg)\,\mathrm{d}x\mathrm{d}\tau\bigg|\\
&+\bigg|\int_{Q_{r_{k-1/3}}}\bigg(\bar{P}_{k2}{w w_k\over |w|}\cdot\nabla \varphi_k+\bar{P}_{k2}\Div\bigg({w w_k\over |w|}\bigg)\varphi_k\bigg)\,\mathrm{d}x\mathrm{d}\tau\bigg|\\
&+\bigg|\int_{Q_{r_{k-1/3}}}\nabla \bar{P}_{k3}\cdot {w w_k\over |w|}\varphi_k\,\mathrm{d}x\mathrm{d}\tau\bigg|\\
\triangleq &\text{I.2.1}+\text{I.2.2}+\text{I.2.3}.
\end{align*}
Due to the definition of $\bar{P}_{k3}$ and $\phi_k$, we have for any $x\in B_{r_{k-2/3}}$
\[|\nabla \tilde{P}_{k3}|\lesssim 2^{4k}\int_{B_{r_{k-2/3}}}\left(|w|^2+|w||v|\right)\,\mathrm{d}y+2^{4k}\int_{B_{r_{k-2/3}}}|\bar{P}|\,\mathrm{d}y.\]
This, together with H\"{o}lder's inequality, yields that
\begin{align*}
\|\nabla \bar{P}_{k3}\|_{L^{{3\over 2}}_tL^{\infty}_x(Q_{r_{k-1/3}})}\lesssim &2^{4k} \|w\|^2_{L^{\infty}_tL^2_x(Q_1)}
+2^{4k}\|w\|_{L^{\infty}_tL^2_x(Q_1)}\|v\|_{L^p_tL^q_x(Q_1)}+2^{4k}\|\bar{P}\|_{L^{{3\over 2}}(Q_1)}\\
\lesssim& 2^{4k}W^2_0+2^{4k}W_0\|v\|_{L^p_tL^q_x(Q_1)}+2^{4k}\|\bar{P}\|_{L^{{3\over 2}}(Q_1)}.
\end{align*}
This together with H\"{o}lder's inequality and Lemma \ref{lem.w_k} entails
\begin{equation}\label{est.2}
\begin{split}
\text{I.2.3}\lesssim& \|w_k\|_{L^{3}_tL^{18/5}_x(Q_{r_{k-1}})}\|\nabla \tilde{P}\|_{L^{{3\over 2}}_tL^{\infty}_x(Q_{r_{k-1/3}})}\|I_{\{w_k>0\}}
\|_{L^{\infty}_tL^{18/13}_x(Q_{r_{k-1}})}\\
\lesssim&_{\nu} 2^{4k+13k/9}W^{11/9}_{k-1}\left(W_0^2+W_0\|v\|_{L^p((-1,0);L^q(B_1))}+\|\bar{P}\|_{L^{{3\over 2}}(Q_1)}\right).
\end{split}
\end{equation}
Now we tackle with the term $\text{I.2.2}$. According to $w=w\big(1-{w_k\over |w|}\big)+{ww_k\over |w|}$, we  rewrite $\bar{P}_{k2}$ as follows:
\begin{align*}
\bar{P}_{k2}(x,t)=&{3\over 4\pi}\int_{\RR^3}\partial_i\partial_j\Big({1\over |x-y|}\Big)\phi_k(y)\Big(w^i\Big(1-{w_k\over |w|}\Big)v^j+w^j\Big(1-{w_k\over |w|}\Big)v^i\Big)(y,t)\,\mathrm{d}y\\
&+{3\over 4\pi}\int_{\RR^3}\partial_i\partial_j\Big({1\over |x-y|}\Big)\phi_k(y)\Big({w^iw_k\over |w|}v^j+{w^jw_k\over |w|}v^i\Big)(y,t)\,\mathrm{d}y\\
\triangleq& \bar{P}_{k21}(x,t)+\bar{P}_{k22}(x,t).
\end{align*}
By Calder\'on-Zygmund estimates, H\"{o}lder's inequality and the fact that $\big|w^i\big(1-{w_k\over |w|}\big)\big|\leq 1$, we immediately have
\[\|\bar{P}_{k21}\|_{L^p_tL^q_x(Q_{r_{k-1/3}})}\lesssim\|v\|_{L^p_tL^q_x(Q_{r_{k-2/3}})}\]and\[
\|\bar{P}_{k22}\|_{L^2_tL^{l_4}_x(Q_{r_{k-1/3}})}\lesssim\|w_k\|_{L^{2p/(p-2)}_tL^{6p/(p+4)}_x(Q_{r_{k-{2/3}}})}
\|v\|_{L^p_tL^q_x(Q_{r_{k-2/3}})}\]
with ${1/l_4}={(p+4)/(6p)}+{1/ q}$.
Hence, by the above two estimates,  H\"{o}lder's inequality, Lemma \ref{lem.link} and Lemma \ref{lem.est}, we get
\begin{equation}\label{est.3}
\begin{split}
\text{I.2.2}\lesssim &2^{k}\|\bar{P}_{k21}\|_{L^p_tL^q_x(Q_{r_{k-1/3}})}
\|w_{k}\|_{L^{2p/(p-1)}_tL^{6p/(p+2)}_x(Q_{r_{k-1}})}\|I_{\{w_{k}>0\}}\|_{L^{2p/(p-1)}_tL^{l_2}_x(Q_{r_{k-1}})}\\
&+\|\bar{P}_{k21}\|_{L^p_tL^q_x(Q_{r_{k-1/3}})}
\|d_k\|_{L^2(Q_{r_{k-1}})}\|I_{\{w_{k}>0\}}\|_{L^{2p/(p-2)}_tL^{2q/(q-2)}_x(Q_{r_{k-1}})}\\
&+2^{k}\|\bar{P}_{k22}\|_{L^2_tL^{l_4}_x(Q_{r_{k-1/3}})}
\|w_{k}\|_{L^2_tL^{6}_x(Q_{r_{k-1}})}\|I_{\{w_{k}>0\}}\|_{L^{\infty}_tL^{l_2}_x(Q_{r_{k-1}})}\\
&+\|\bar{P}_{k22}\|_{L^2_tL^{l_4}_x(Q_{r_{k-1/3}})}
\|d_k\|_{L^2(Q_{r_{k-1}})}\|I_{\{w_{k}>0\}}\|_{L^{\infty}_tL^{l_3}_x(Q_{r_{k-1}})}\\
\lesssim&_{\nu}\bigg(2^{({10\over 3}-{4\over 3p}-{2\over q})k}W^{{5\over 3}-{2\over 3p}-{1\over q}}_{k-1}+2^{({5\over 3}-{4\over 3p}-{2\over q})k}W^{{4\over 3}-{2\over 3p}-{1\over q}}_{k-1}\\
&+2^{({7\over 3}-{4\over 3p}-{2\over q})k}W^{{5\over 3}-{2\over 3p}-{1\over q}}_{k-1}+2^{({2\over 3}-{4\over 3p}-{2\over q})k}W^{{4\over 3}-{2\over 3p}-{1\over q}}_{k-1}\bigg)\|v\|_{L^p_tL^q_x(Q_1)}.
\end{split}
\end{equation}
Finally, we tackle with the term $\text{I.2.1}$. For $\tilde{P}_{k1}$, we can split it into three terms:
\begin{align*}
\bar{P}_{k1}=&{3\over 4\pi}\int_{\RR^3}\partial_i\partial_j\bigg({1\over |x-y|}\bigg)\phi_kw^i\Big(1-{w_k\over |w|}\Big)w^j\bigg(1-{w_k\over |w|}\bigg)\,\mathrm{d}y\\
&+{3\over 2\pi}\int_{\RR^3}\partial_i\partial_j\bigg({1\over |x-y|}\bigg)\phi_kw^i\Big(1-{w_k\over |w|}\bigg){w^j w_k\over |w|}\,\mathrm{d}y\\
&+{3\over 4\pi}\int_{\RR^3}\partial_i\partial_j\bigg({1\over |x-y|}\bigg)\phi_k{w^i w_k\over |w|}{w^j w_k\over |w|}\,\mathrm{d}y\\
\triangleq& \bar{P}_{k11}+\bar{P}_{k12}+\bar{P}_{k13}.
\end{align*}
By Calder\'on-Zygmund estimates and Lemma \ref{lem.link}, one has
\[\|\bar{P}_{k11}\|_{L^s(Q_{r_{k-1}})}\leq C(s)\,\, \forall 1<s<\infty, \quad \|\bar{P}_{k12}\|_{L^{{10\over 3}}(Q_{r_{k-1/3}})}\leq C\|w_k\|_{L^{{10\over 3}}(Q_{k-{2/3}})}.\]
Therefore, by H\"{o}lder's inequality, Lemma \ref{lem.est} and Lemma \ref{lem.w_k}, we have
\begin{equation}\label{est.4}
\begin{split}
&\bigg|\int_{Q_{r_{k-1/3}}}(\bar{P}_{k11}+\bar{P}_{k12}) w\cdot\nabla \varphi_k-\nabla (\bar{P}_{k11}+\bar{P}_{k12})\cdot w\bigg({w_k\over |w|}-1\bigg)\varphi_k\,\mathrm{d}x\mathrm{d}\tau\bigg|\\
=& \bigg|\int_{Q_{r_{k-1/3}}}(\bar{P}_{k11}+\bar{P}_{k12}){w w_k\over |w|}\cdot\nabla \varphi_k+(\bar{P}_{k11}+\bar{P}_{k12})\Div\bigg({w w_k\over |w|}\bigg)\varphi_k\,\mathrm{d}x\mathrm{d}\tau\bigg|\\
\lesssim &2^k\|\bar{P}_{k11}+\bar{P}_{k12}\|_{L^{{10\over 3}}(Q_{r_{k-1/3}})}\|w_k\|_{L^{{10\over 3}}(Q_{r_{k-1}})}\|I_{\{w_k>0\}}\|_{L^{{5\over 2}}(Q_{r_{k-1}})}\\
&+\|\bar{P}_{k11}\|_{L^{6}(Q_{r_{k-1/3}})}\|d_k\|_{L^{2}(Q_{r_{k-1}})}
\|I_{\{w_k>0\}}\|_{L^{3}(Q_{r_{k-1}})}\\
&+\|\bar{P}_{k12}\|_{L^{{10\over 3}}(Q_{r_{k-1/3}})}\|d_k\|_{L^{2}(Q_{r_{k-1}})}\|I_{\{w_k>0\}}\|_{L^{5}(Q_{r_{k-1}})}\\
 \lesssim&_{\nu} 2^{7k/6}(W^{7/6}_{k-1}+W^{5/3}_{k-1})+2^{10k/9}W^{19/18}_{k-1}+2^{2k/3}W^{4/3}_{k-1}.
\end{split}
\end{equation}
For $\bar{P}_{k13}$, we have
\begin{align*}
\nabla \bar{P}_{k13}=&{3\over 4\pi}\int_{\RR^3}\partial_i\partial_j\bigg({1\over |x-y|}\bigg)\nabla\phi_k{w^i w_k\over |w|}{w^j w_k\over |w|}\,\mathrm{d}y\\
&+{3\over 2\pi}\int_{\RR^3}\partial_i\partial_j\bigg({1\over |x-y|}\bigg)\phi_k\nabla\bigg({w^i w_k\over |w|}\bigg){w^j w_k\over |w|}\,\mathrm{d}y\\
\triangleq& J_1+J_2.
\end{align*}
By Calder\'on-Zygmund estimates and Lemma \ref{lem.w_k}, we obtain that
\[\|\bar{P}_{k13}\|_{L^{5/3}(Q_{r_{k-1/3}})}\lesssim \|w_k\|^2_{L^{10/3}(Q_{r_{k-2/3}})},\quad \|J_1\|_{L^{{5\over 3}}(Q_{r_{k-1/3}})}\lesssim 2^k\|w_k\|^2_{L^{10/3}(Q_{r_{k-2/3}})},\]
and
\[\|J_2\|_{L^{{5\over 4}}(Q_{r_{k-1/3}})}\lesssim \|w_k\|_{L^{10/3}(Q_{r_{k-2/3}})}\|d_k\|_{L^{2}(Q_{r_{k-2/3}})}.\]
Hence, we have
\begin{equation}\label{est.5}
\begin{split}
&\bigg|\int_{Q_{r_{k-1/3}}}\bar{P}_{k13} w\cdot\nabla \varphi_k-\nabla \bar{P}_{k13}\cdot w\bigg({w_k\over |w|}-1\bigg)\varphi_
k\,\mathrm{d}x\mathrm{d}\tau\bigg|\\
\lesssim& 2^k \|\bar{P}_{k13}\|_{L^{5/3}(Q_{r_{k-1/3}})}\|I_{\{w_k>0\}}\|_{L^{5/2}(Q_{k-1})}
+\|J_1\|_{L^{5/3}(Q_{r_{k-1/3}})}\|I_{\{w_k>0\}}\|_{L^{5/2}(Q_{r_{k-1}})}\\
&+2^k\|\bar{P}_{k13}\|_{L^{5/3}(Q_{r_{k-1/3}})}\|w_k\|_{L^{10/3}(Q_{r_{k-1}})}
\|I_{\{w_k>0\}}\|_{L^{10}(Q_{r_{k-1}})}
\\
&+\|J_2\|_{L^{5/4}(Q_{r_{k-1/3}})}\|I_{\{w_k>0\}}\|_{L^{5}(Q_{r_{k-1}})}\\
 \lesssim&_{\nu} 2^{7k/3}W^{5/3}_{k-1}+2^{4k/3}W^{5/3}_{k-1}+2^{2k/3}W^{4/3}_{k-1}.
\end{split}
\end{equation}
Collecting estimates \eqref{est.1}-\eqref{est.5}, we can say there exist a $C_2(\nu)>1$ and $\beta>1$, such that, for any $k>0$
\begin{equation}\label{w-iter}
W_k\leq C^k_2(\nu)W^{\beta}_{k-1}.
\end{equation}
\textbf{Step 3. Conclusion.}

From \eqref{est.0}, we have $W_0<1$ which implies $W_k<1$ for any $k\geq 0$. So,
Let $\bar{W}_k=C_2^{{k\over \beta-1}}(\nu)C_2^{{1\over (\beta-1)^2}}(\nu)W_k$. Then, from \eqref{w-iter}, we have $\bar{W}_k\leq \bar{W}_{k-1}.$ If we choose $\epsilon_1$ small enough such that $\bar{W}_0\leq C^{{k\over \beta-1}}_2(\nu)C^{{1\over (\beta-1)^2}}_2(\nu)\epsilon_1<1$, then we have $\bar{W}_k\leq 1$ for any $k\geq 0$, which implies that, for any integer $k\geq 0$,
\[W_k\leq C^{-{k\over \beta-1}}_2(\nu)C^{-{1\over (\beta-1)^2}}_2(\nu).\]
Thanks to $C_2(\nu)>1$, sending $k\to\infty$, we get $W_k\to 0$. So we prove \eqref{w-bound}.

\subsection{Stability of singularities}
In this part, we will show the stability of singularities of suitable weak solutions to  \eqref{eq.W'} in the sense of locally strong limits.

First, we introduce   definition of singularity. Given a suitable weak solution $(w,\bar{P})$ to  \eqref{eq.W'}, if $w$ is essentially bounded in a neighborhood of $z_0=(x_0,t_0)\in D$, we call $z_0$ is a regular point of $(w,\bar{P})$; otherwise, we call it a singular point.

The main result in this subsection can be stated as follows:
\begin{proposition}[Stability of singularities]\label{pro.stability}
Let  $v^{(k)}\in L^p_tL^q_x(Q_1)$  satisfies $\Div v^{(k)}=0$ with $\frac2p+\frac3q<1$ and $q>3$. Let $(w^{(k)},\bar{P}^{(k)})$ be a sequence of suitable weak solutions to the following equations  in $Q_1$
\begin{equation}\label{eq.Wk}
\left\{\begin{array}{ll}
\partial_t w^{(k)}-\nu\Delta w^{(k)}+w^{(k)}\cdot\nabla w^{(k)}+v^{(k)}\cdot\nabla w^{(k)}+w^{(k)}\cdot\nabla v^{(k)}+\nabla \bar{P}^{(k)}=0\\
\Div w^{(k)}=0.
\end{array}\right.
\end{equation}
Assume further that
\[w^{(k)}\rightarrow w \,\,\text{ in }\,\,L^3(Q_1),\quad \bar{P}^{(k)}\rightharpoonup \bar{P}\,\, \text{ in }\,\, L^{3/2}(Q_1),\quad v^{(k)}\rightharpoonup v \,\,\text{ in }\,\,L^p_tL^q_x(Q_1),\]
  and the limit $(w,\bar{P})$ is a suitable weak solution to Eq. \eqref{eq.W'} in $Q_1$ associated to $v$. Then, if  $z^{(k)}\in Q_1$ is a singular point of $(w^{(k)},\bar{P}^{(k)})$ and $z^{(k)}\rightarrow z_0\in Q_1$, we have  $z_0$ is a singular point of $(w,\bar{P})$.
\end{proposition}
\begin{proof}
We will prove the proposition by contradiction. Assume $z_0\in Q_1$ is a regular point of $(w,\bar{P})$. Then, there exists $r_0>0$ and $M_0>0$ such that $Q_{r_0}(z_0)\subset Q_1$ and  $\|w\|_{L^{\infty}(Q_{r_0}(z_0))}\leq M_0$. Thanks to  that fact that  $w^{(k)}\rightarrow w$ in $L^3(Q_r(z_0))$, there exists $N_1>0$ such that, for any $k\geq N_1$ and $r_1\leq r\leq r_0$, $r_1<{r_0/2}$,
\begin{equation}\label{eq.w^k}
\begin{split}
r^{-2}\int_{Q_r(z_0)} |w^{(k)}|^3\,\mathrm{d}x\mathrm{d}\tau\leq &C_1r^{-2}\int_{Q_r(z_0)} |w^{(k)}-w|^3\,\mathrm{d}x\mathrm{d}\tau+C_1r^{-2}\int_{Q_r(z_0)} |w|^3\,\mathrm{d}x\mathrm{d}\tau\\
\leq&C_1M^3_0r_0^3.
\end{split}
\end{equation}

For $\bar{P}^{(k)}$, from equations \eqref{eq.W'}, we have
$$-\Delta \bar{P}^{(k)}=\Div\Div (w^{(k)}\otimes w^{(k)}+v^{(k)}\otimes w^{(k)}+w^{(k)}\otimes v^{(k)})\quad \text{in}\,\, Q_1.$$
Hence, we can write $\bar{P}^{(k)}=g^{(k)}+h^{(k)}$ in $Q_{r}(z_0)$ with $r_1\leq r\leq r_0$, where
\[g^{(k)}=(-\Delta)^{-1}\Div\Div \left((w^{(k)}\otimes w^{(k)}+v^{(k)}\otimes w^{(k)}+w^{(k)}\otimes v^{(k)})I_{B_{r}(x_0)}\right)\]
and $h^{(k)}\in L^{3/2}(Q_{r}(z_0))$ is harmonic in $B_{r}(x_0)$.

 By Calder\'{o}n-Zygmund estimates, H\"{o}lder's inequality and \eqref{eq.w^k}, we get that for any $k>N_1$ and $r_1\leq r_0$,
\begin{equation}\label{eq.g^k1}
\begin{split}
&r^{-2}\|g^{(k)}\|^{3/2}_{L^{3/2}(Q_r(z_0))}\\
\leq& Cr^{-2}\|w^{(k)}\|^3_{L^{3}(Q_r(z_0))}+Cr^{-2}\|w^{(k)}\|^{3/2}_{L^3(Q_r(z_0))}
\|v^{(k)}\|^{3/2}_{L^p_tL^q_x(Q_r(z_0))}\|I\|_{L^{\frac{p-3}{3p}}_tL^{\frac{q-3}{3q}}_x(Q_r(z_0))}\\
\leq& C_2M^3_0r^3_0+C_2M_0^{3/2}r_0^{1/ 2}\|v^{(k)}\|^{3/2}_{L^p_tL^q_x(Q_r(z_0))}
\end{split}
\end{equation}
when $p\geq 3$, and
\begin{equation}\label{eq.g^k2}
\begin{split}
r^{-2}\|g^{(k)}\|^{3/2}_{L^{3/2}(Q_r(z_0))}\leq& Cr^{-2}\|w^{(k)}\|^3_{L^{3}(Q_r(z_0))}+Cr^{-2}\|I\|^{3/2}_{L^{2p/(p-2)}_tL^{4q/(q-4)}_x(Q_r(z_0))}\\
&\times\|w^{(k)}\|^{3/4}_{L^3(Q_r(z_0))}\|w^{(k)}\|^{3/4}_{L^{\infty}_tL^2_x(Q_r(z_0))}\|v^{(k)}\|^{3/2}_{L^p_tL^q_x(Q_r(z_0))}\\
\leq& C_3M^3_0r^3_0+C_3 r_0^{15/8-3/p-9/(2q)}M^{3/4}_0\|v^{(k)}\|^{3/2}_{L^p_tL^q_x(Q_r(z_0))}
\end{split}
\end{equation}
when $2<p<3$.

Collecting estimates \eqref{eq.w^k}-\eqref{eq.g^k2}, we can choose $r_0$ small enough such that, for any $k\geq N_1$ and $r_1\leq r\leq r_0$,
\begin{equation}\label{eq.5}
r^{-2}\int_{Q_r(z_0)}\left(|w^{(k)}|^3+|g^{(k)}|^{3/2}\right)\,\mathrm{d}x\mathrm{d}\tau<{\varepsilon_0/2}.
\end{equation}
where $\varepsilon_0$ is the constant in Theorem \ref{regularity}.

For $h^{(k)}$, by the mean value theorem, we have for any $r_1\leq r\leq {r_0\over 2}$,
\begin{align*}
\|h^{(k)}(t)\|_{L^{\infty}(B_r(x_0))}\leq C (r_0)^{-3}\int_{B_{r_0}(x_0)}|h^{(k)}(t)|\,\mathrm{d}x\leq C(r_0)^{-2}\bigg(\int_{B_{r_0}(x_0)}|h^{(k)}(t)|^{{3\over 2}}\,\mathrm{d}x\bigg)^{2/3}.
\end{align*}
Hence, for any $r_1\leq r<{r_0/2}$, we have
\begin{align*}
r^{-2}\int_{Q_{r}(z_0)}|h^{(k)}|^{3\over 2}\,\mathrm{d}x\mathrm{d}\tau\leq C(r_0)^{-3} r\int_{Q_{r_0}(x_0)}|h^{(k)}|^{{3\over 2}}\,\mathrm{d}x\mathrm{d}\tau.
\end{align*}
Now we choose $r_1$ small enough, such that
\[C(r_0)^{-3} r_1\int_{Q_{r_1}(x_0)}|h^{(k)}|^{{3\over 2}}\,\mathrm{d}x\mathrm{d}\tau<\varepsilon_0/2.\]
This together with \eqref{eq.5} and  Proposition \ref{regularity} enables us to conclude that
$(w^{(k)},\bar{P}^{(k)})$ is regular in $Q_{r_1}(z_0)$, which is a contradiction. So, we complete the proof of Proposition~\ref{pro.stability}.
\end{proof}
\subsection{The short-time behaviour of the kinetic energy }
In this subsection, we will give an useful observation for  energy solutions to equations~\eqref{W-2} where $v$ is the mild solution stated in the beginning of Section \ref{sec.4}. Here, we adapt the notation:
$$Y^s_{p,q}(T)\triangleq \widetilde{L}^{\infty}([0,T];\,F\dot{B}^{s}_{p,q}(\RR^3))\cap \widetilde{L}^1([0,T];\,F\dot{B}^{2+s}_{p,q}(\RR^3)),$$
and
\[\hat{K}^{s}_p(T)\triangleq\Big\{f(t)\in \mathcal{S}'(\RR^d)\big|\,\lim_{t\to 0+}t^{-\frac s2}\|f(t)\|_{FL^{p}}=0,\,\|f\|_{\hat{K}^{s}_p(T)}\triangleq
\sup_{0\leq t\leq T}t^{-\frac s2}\|f(t)\|_{FL^{p}}<\infty\Big\}.\]
\begin{proposition}\label{pro.sma}
Assume that $\theta,\,s,\,s_p,\,p,\,\tilde{p},\,q,\,\tilde{q}$ satisfy the relation \eqref{RC}. Let $w$ be a local energy solution to equations \eqref{W-2} associated to initial data $w_0\in L^2(\RR^3)\cap(F\dot{B}^{s_p}_{p,q}(\RR^3)+F\dot{B}^{s}_{\tilde{p},\tilde{q}}(\RR^3))$. Then there exists a $\gamma\in(0,{s-s_{\tilde{p}}\over 2\theta})$ depending on $\|w_0\|_{L^2\cap (F\dot{B}^{s_p}_{p,q}+F\dot{B}^{s}_{\tilde{p},\tilde{q}})}$ such that for any $0<t\ll 1$,
\begin{equation}\label{W}
\|w(t)-e^{\nu t\Delta}w_0\|_{L^2} \leq C\big(\|w_0\|_{L^2\cap (F\dot{B}^{s_p}_{p,q}+F\dot{B}^{s}_{\tilde{p},\tilde{q})}(\RR^3)},
\|v\|_{Y^s_{\tilde{p},\tilde{q}}(T)}\big)\left(t^{\frac{s-s_{\tilde{p}}-\gamma \theta}{ 2}}+t^{\gamma(1-\theta)}\right),
\end{equation}

\end{proposition}
\begin{proof}
Firstly, we split $w_0$ into $f_0+g_0$ satisfying $f_0\in F\dot{B}^{s}_{\tilde{p},\tilde{q}}(\RR^3)$ and $g_0\in L^2(\RR^3)$. Let $$\alpha=\|w_0\|_{L^2\cap (F\dot{B}^{s_p}_{p,q}+F\dot{B}^{s}_{\tilde{p},\tilde{q}})}.$$
From the definition of $F\dot{B}^{s_p}_{p,q}(\RR^3)+F\dot{B}^{s}_{\tilde{p},\tilde{q}}(\RR^3)$, we can choose a decomposition of $w_0$: $w_0=\tilde{f}_0+\tilde{g}^1_0$ such that
\[\|\tilde{f}_0\|_{F\dot{B}^{s_p}_{p,q}}+\|\tilde{g}^1_0\|_{F\dot{B}^{s}_{\tilde{p},\tilde{q}}}\leq 2\alpha.\]
Since $\Div w_0=0$, we have $$w_0=Pw_0=\mathbb{P}\tilde{f}_0+\mathbb{P}\tilde{g}^1_0\triangleq f_0+g^1_0.$$By Calder\'{o}n-Zygmund estimates, we readily get
\[\|f_0\|_{F\dot{B}^{s_p}_{p,q}}+\|g^1_0\|_{F\dot{B}^{s}_{\tilde{p},\tilde{q}}}\leq C\alpha.\]
Then, by Lemma \ref{lem.decom}, for any $j\in \ZZ$, $f_0$ can be decomposed as $f_0=g^2_0+h_0$ such that $\Div g^2_0=\Div h_0=0$ and
\[\|g^2_0\|_{F\dot{B}^{s}_{\tilde{p},\tilde{q}}}\leq C2^{-j\theta}\|f_0\|_{F\dot{B}^{s_p}_{p,q}}\leq C2^{-j\theta}\alpha,\quad \|h_0\|_{L^2}\leq C2^{j(1-\theta)}\|f_0\|_{F\dot{B}^{s_p}_{p,q}}\leq C2^{j(1-\theta)}\alpha.\]
where $\theta,\,p,\,\tilde{p},\,q,\,\tilde{q},\,s,\,s_p$ satisfy \eqref{RC}.

Next, we consider the following problem:
\begin{equation}\label{Eq.g}
\partial_t g-\nu\Delta g+g\cdot\nabla g+v\cdot\nabla g+g\cdot\nabla v+\nabla \bar{P}=0,\quad \Div g=0,
\end{equation}
which is supplemented with initial condition $g(x,0)=g_0(x)$.

By Duhamel formula, one writes
\begin{align*}
g=&e^{\nu t\Delta}g_0+\int^t_0e^{-\nu(t-s)\Delta}\mathbb{P}(g\cdot\nabla v+v\cdot\nabla g)\,\mathrm{d}\tau+\int^t_0e^{-\nu(t-s)\Delta}\mathbb{P}(g\cdot\nabla g)\,\mathrm{d}\tau\\
\triangleq& G+L(g)+B(g,g).
\end{align*}
Following the proof of Theorem \ref{well-sub}, we have that
\[\|G\|_{Y^s_{\tilde{p},\tilde{q}}(T_1)}\leq C(\nu)\|g_0\|_{F\dot{B}^s_{p,q}},\quad \|L(g)\|_{Y^s_{\tilde{p},\tilde{q}}(T_1)}\leq C(\nu)T_1^{\frac{s-s_{\tilde{p}}}2}\|g\|_{Y^s_{\tilde{p},\tilde{q}}(T_1)}\|v\|_{Y^s_{\tilde{p},\tilde{q}}(T_1)},\]
and
\[\|B(g,\bar{g})\|_{Y^s_{\tilde{p},\tilde{q}}(T_1)}
\leq C(\nu)T_1^{\frac{s-s_{\tilde{p}}}2}\|g\|_{Y^s_{\tilde{p},\tilde{q}}(T_1)}\|\bar{g}\|_{Y_{Y^s_{\tilde{p},\tilde{q}}(T_1)}}.\]
Thus, by Lemma \ref{fixed}, there exists a mild solution $g$ to problem \eqref{Eq.g} on $[0,T_1]$ corresponding to initial data $g_0$, satisfying
\[g\in C_b\big([0,T_1];\,F\dot{B}^s_{\tilde{p},\tilde{q}}(\RR^3)\big)\cap Y^s_{\tilde{p},\tilde{q}}(T_1).\]
Furthermore, $T_1$ and $g$ satisfy
\[T_1\leq \min\Big\{1,T, \big(C(\nu)(\|v\|_{Y^s_{\tilde{p},\tilde{q}}(T)}
+2^{-j\theta}\alpha)\big)^{-\frac{2}{s-s_{\tilde{p}}}}\Big\}\]
and
\begin{equation}\label{est.g}
\|g\|_{Y_{T_1}}\leq\frac{C(\nu)\|g_0\|_{F\dot{B}^s_{\tilde{p},\tilde{q}}}}{ 1-C(\nu)T^{(s-s_p)/2}_1\|v\|_{Y^s_{\tilde{p},\tilde{q}}(T)}}\leq \frac{C(\nu) 2^{-j\theta}\alpha}{1-C(\nu)T^{(s-s_p)/2}_1\|v\|_{Y^s_{\tilde{p},\tilde{q}}(T)}}.
\end{equation}
In the same way as in the proof of Proposition \ref{pro.ana}, by Lemma \ref{regul}, we have
\begin{align*}
&\|te^{\nu t\Delta} g_0\|_{\widetilde{L}^{\infty}([0,T_1];\,F\dot{B}^{s+2}_{\tilde{p},\tilde{q}})}\leq C(\nu) \|g_0\|_{F\dot{B}^{s}_{\tilde{p},\tilde{q}}}<\infty,
\end{align*}
\begin{align*}
\|tL(g)\|_{\widetilde{L}^{\infty}([0,T_1];\,F\dot{B}^{s+2}_{\tilde{p},\tilde{q}})}
\leq& C(\nu) T_1^{\frac{s-s_{\tilde{p}}}2}\Big(\|tg\|_{\widetilde{L}^{\infty}([0,T_1];\,F\dot{B}^{s+2}_{\tilde{p},\tilde{q}})}
+\|g\|_{Y^s_{\tilde{p},\tilde{q}}(T_1)}\Big)\|v\|_{Y^s_{\tilde{p},\tilde{q}}(T_1)}
\end{align*}
and
\begin{align*}
&\max\left\{ \|tB(g,\bar{g})\|_{\widetilde{L}^{\infty}([0,T_1];\,F\dot{B}^{s+2}_{\tilde{p},\tilde{q}})},\,
\|tB(\bar{g},g)\|_{\widetilde{L}^{\infty}([0,T_1];\,F\dot{B}^{s+2}_{\tilde{p},\tilde{q}})}\right\}\\
\leq& C(\nu) T_1^{\frac{s-s_{\tilde{p}}}2}\Big(\|tg\|_{\widetilde{L}^{\infty}([0,T_1];\,F\dot{B}^{s+2}_{\tilde{p},\tilde{q}})}
+\|g\|_{Y^s_{\tilde{p},\tilde{q}}(T_1)}\Big)\|\bar{g}\|_{Y^s_{\tilde{p},\tilde{q}}(T_1)}.
\end{align*}
Employing Lemma \ref{regul}, we get $$\|tg(t)\|_{\widetilde{L}^{\infty}([0,T_1];F\dot{B}^s_{\tilde{p},\tilde{q}})}\leq \frac{C (\nu)2^{-j\theta}\alpha}{1-C(\nu)T^{(s-s_p)/2}_1\|v\|_{Y^s_{\tilde{p},\tilde{q}}(T_1)}}.$$
This, together with Remark \ref{rem.ana}, we immediately obtain that
\begin{equation}\label{est.g'}
\|g\|_{\hat{K}^s_{\tilde{p}}(T_1)}\leq \frac{C (\nu)2^{-j\theta}\alpha}{1-C(\nu)T^{(s-s_p)/2}_1\|v\|_{Y^s_{\tilde{p},\tilde{q}}(T_1)}}.
\end{equation}
On the other hand, due to $g_0=w_0-h_0\in L^2(\RR^3)$, we have
$$\|e^{\nu t\Delta}g_0\|_{L^2}\in L^2(\RR^3).$$
Moreover, by H\"{o}lder's and Young's inequalities, we have the following estimate (we omit the routine calculation):
\[\|L(g)\|_{L^{\infty}((0,T_1);L^2)}\leq C(\nu)T^{\frac{s-s_{\tilde{p}}}{2}}\|g\|_{L^{\infty}((0,T_1);L^2)}\|v\|_{\hat{K}^s_{\tilde{p}}(T)}\]
and
\[\max\left\{\|B(\bar{g},g)\|_{L^{\infty}((0,T_1);L^2)},\,\|B(g,\bar{g})\|_{L^{\infty}((0,T_0);L^2)}\right\}
\leq C(\nu)T^{\frac{s-s_{\tilde{p}}}{2}}\|g\|_{L^{\infty}((0,T_1);L^2)}\|\bar{g}\|_{\hat{K}^s_{\tilde{p}}(T_1)}.\]
Hence, by Lemma \ref{regul}, we obtain that $g\in C([0,T_1];L^2(\RR^3))$.

Similarly, we can deduce that $g\in L^2((0,T_1);\dot{H}^1)$. So, we see that $g$ satisfies the global energy inequality~\eqref{glo.W} on $(0,T_1)$. We have by  H\"{o}lder's, Hausdroff-Young's and Young's inequalities that  for any $0<t<T_1$
\begin{align*}
&\int_{\RR^3\times\{t\}}|g|^2\,\mathrm{d}x+2\int^t_0\int_{\RR^3}\nu|\nabla g|^2\,\mathrm{d}x\mathrm{d}\tau\\
\leq& \|g_0\|^2_{L^2}+C\int^t_0\|g\|_{L^{{2\tilde{p}\over 2-\tilde{p}}}}\|\nabla g\|_{L^2}\|v\|_{L^{{\tilde{p}\over \tilde{p}-1}}}\,\mathrm{d}\tau\\
 \leq& \|g_0\|^2_{L^2}+\int^t_0\|g\|^{{3-2\tilde{p}\over \tilde{p}}}_{L^2}\|\nabla g\|^{4-{3\over \tilde{p}}}_{L^2}\|v\|_{FL^{\tilde{p}}}\,\mathrm{d}\tau\\
 \leq& \|g_0\|^2_{L^2}+\int^t_0\int_{\RR^3}\nu |\nabla g|^2\,\mathrm{d}x\mathrm{d}\tau
+C(\nu)\int^t_0\|g\|^2_{L^2}\|v\|^{{2\tilde{p}\over 3-2\tilde{p}}}_{(FL^{\tilde{p}})^2}\,\mathrm{d}\tau.
\end{align*}
In the third line, we have used  Sobolev embedding \begin{equation}\label{sobolev}
\dot{H}^1(\RR^3)\hookrightarrow L^6(\RR^3)
\end{equation}
and interpolation inequality
\begin{equation}\label{intero}
\|g\|_{L^{2\tilde{p}\over 2-\tilde{p}}}\leq \|g\|^{{3-2\tilde{p}\over \tilde{p}}}_{L^2}\|g\|^{3-{3\over \tilde{p}}}_{L^6}.
\end{equation}
Finally, by Gronwall's inequality, we get for any $0<t\leq T_1$
\begin{equation*}
\sup_{0\leq \tau\leq t}\|g\|^2_{L^2}+\nu\int^{t}_0\|\nabla g\|^2_{L^2}\,\mathrm{d}\tau\leq \|g_0\|^2_{L^2}e^{C(\nu){t}^{1-{s\over s_{\tilde{p}}}} \|v\|_{\hat{K}^s_{\tilde{p}}(T)}^{2\tilde{p}\over 3-2\tilde{p}}}.
\end{equation*}
With this property and \eqref{est.g'} in hand, we have that for any $0<t<T_1$,
\begin{equation}\label{g}
\begin{split}
&\|g(t)-e^{\nu t\Delta}g_0\|_{L^2}\\
\leq& Ct^{{s-s_{\tilde{p}}\over 2}}\|g\|_{L^{\infty}((0,t);L^2)}\big(\|g\|_{\hat{K}^s_{\tilde{p}}(T_1)}+
\|v\|_{\hat{K}^s_{\tilde{p}}(T_1)}\big)\\
\leq&Ct^{{s-s_{\tilde{p}}\over 2}}\|g_0\|_{L^2}e^{C(\nu){t}^{1-{s\over s_{\tilde{p}}}}\|v\|_{\hat{K}^s_{\tilde{p}}(T)}^{2\tilde{p}\over 3-2\tilde{p}}}
\bigg(C\frac{2^{-j\theta}\alpha}{1-C(\nu)T^{(s-s_p)/2}_1 \|v\|_{Y^s_{\tilde{p},\tilde{q}}(T)}}+
\|v\|_{\hat{K}^s_{\tilde{p}}(T)}\bigg).
\end{split}
\end{equation}
Letting $h=w-g$, we have $h\in L^{\infty}((0,T_1);L^2)\cap L^2((0,T_1);\dot{H}^1)$ and $\displaystyle\lim_{t\to 0+}\|h(t)-h_0\|_{L^2}=0$.

Thanks to \eqref{est.g'}, we get $g\in L^{r}((0,T_2);L^{\tilde{p}'})$ with $\frac{2}{r}+\frac{3}{\tilde{p}'}=1$ and $\tilde{p}'$ is the conjugate index of $\tilde{p}$. Thus, using the usual mollification procedure used in Proposition \text{14.3} in \cite{LRbook}, we get the following equality:
\begin{align*}
&\int_{\RR^3} g(t,x)\cdot w(t,x)\,\mathrm{d}x+2 \nu\int^t_{0}\int_{\RR^3}\nabla g:\nabla w\,\mathrm{d}x\mathrm{d}\tau\\
=&-\int^t_{0}\int_{\RR^3}\left((h \cdot\nabla)w\cdot g+ (g\cdot\nabla)w\cdot v+(w\cdot\nabla)g\cdot v\right)\,\mathrm{d}x\mathrm{d}\tau+\int^t_0\int_{\RR^3}g_0\cdot w_0\,\mathrm{d}x.
\end{align*}
This, together with the global energy inequality \eqref{loc.W} for $w$ and $g$, gives that
\begin{align*}
\|h(t)\|^2_{L^2}+2\nu\int^t_0\|\nabla h\|^2_{L^2}\,\mathrm{d}\tau\leq\|h_0\|^2_{L^2}+2\int^t_0\int_{\RR^3}\big((h\cdot\nabla)h\cdot g+(h\cdot\nabla)h\cdot v\big)\,\mathrm{d}x\mathrm{d}\tau.
\end{align*}
By H\"{o}lder's and Young's inequality, \eqref{sobolev} and \eqref{intero}, we get
\[\int_{\RR^3\times\{t\}}|h|^2\,\mathrm{d}x+\int^t_0\int_{\RR^3}\nu|\nabla h|^2\,\mathrm{d}x\mathrm{d}\tau
\leq \|h_0\|^2_{L^2}+C(\nu)\int^t_0\|h\|^2_{L^2}\bigg(\|g\|^{{2\tilde{p}\over 3-2\tilde{p}}}_{FL^{\tilde{p}}}+\|v\|^{{2\tilde{p}\over 3-2\tilde{p}}}_{FL^{\tilde{p}}}\bigg)\,\mathrm{d}\tau.\]
Employing the Gronwall inequality to the above inequality, we obtain that
\begin{equation}\label{h}
\begin{split}
\|h(t)\|^2_{L^2}+\nu\int^{t}_0\|\nabla h\|^2_{L^2}\,\mathrm{d}s\leq& \|h_0\|^2_{L^2}e^{C(\nu)t^{1-{s\over s_{\tilde{p}}}}\big(\|g\|_{\hat{K}^s_{\tilde{p}}(T_1)}+
\|v\|_{\hat{K}^s_{\tilde{p}}(T)}\big)^{2\tilde{p}\over 3-2\tilde{p}}}\\
\leq &2^{2(1-\theta)j}e^{C(\nu)t^{1-{s\over s_{\tilde{p}}}}\big(
\|v\|_{\hat{K}^s_{\tilde{p}}(T_1)}+C(T_1,
\|v\|_{Y^s_{\tilde{p},\tilde{q}}(T)})2^{-j\theta}\alpha\big)^{2\tilde{p}\over 3-2\tilde{p}}}.
\end{split}
\end{equation}
Now, we choose $2^{j}=t^{\gamma}$ for some positive constant $\gamma$ and $0<t<T_1$, which implies that
\[0<\gamma<{s-s_{\tilde{p}}\over 2\theta}\quad\text{and}\quad 0<t<(C(\nu)\alpha)^{-{2\over s-s_{\tilde{p}}-2\theta\gamma}}.\]
According to \eqref{est.g}, \eqref{g} and\eqref{h}, for any $0<t<\min\big\{T_1,(C^{1/ 2}(\nu)\alpha)^{-{2\over s-s_{\tilde{p}}-2\theta\gamma}}\big\}$, we have
\begin{align*}
\|w(t)-e^{\nu\Delta t}w_0\|_{L^2}\leq& \|g(t)-e^{\nu\Delta t}g_0\|_{L^2}+\|h(t)\|_{L^2}+\|e^{\nu\Delta t}h_0\|_{L^2}\\
\leq&C(T_1,T,\alpha, \|v_0\|_{F\dot{B}^s_{\tilde{p},\tilde{q}}})\bigg(t^{\frac{s-s_{\tilde{p}}-\gamma\theta}{ 2}}+t^{\gamma(1-\theta)}\bigg).
\end{align*}
This completes the proof.
\end{proof}

\subsection{Weak-strong uniqueness}
This section is devoted to  the study of a new version of ``weak-strong" uniqueness of weak energy solution problem \eqref{W-2} established at the beginning of Section 4. Here, we still adopt the notation $Y^s_{p,q}(T)$ introduced in above subsection.
\begin{proposition}\label{pro.uni'}
Suppose $s,s_p,p,q,\tilde{p},\tilde{q}$ satisfy relation \eqref{RC}. Let $w$ and $\tilde{w}$ both are weak energy solutions to Eq. \eqref{W-2} with the same initial data $w_0\in L^2(\RR^3)\cap (F\dot{B}^{s_p}_{p,q}(\RR^3)+F\dot{B}^s_{\tilde{p},\tilde{q}}(\RR^3))$. Assume further that \begin{equation}\label{est.w}
w\in \hat{K}^{s_p}_{p}(T)+\hat{K}^s_{\tilde{p}}(T).
\end{equation}
Then $\tilde{w}\equiv w$ on $\RR^3\times \big(0,T)$.
\end{proposition}
\begin{remark}
Let us point out that the uniqueness result established  by Gallagher and Planchon~\cite{GP02} indicates  Proposition \ref{pro.uni'} when ${2\over q}+{3\over p'}\geq 1$. However, for ${2\over q}+{3\over p'}< 1$,  the argument in  \cite{GP02} doesn't seem to work. To overcome it, we need resort to the
regularity of such weak energy solution in short time which is established in Proposition \ref{pro.sma}.
\end{remark}
\begin{proof}[Proof of Proposition \ref{pro.uni'}]
Let $\delta w=\tilde{w}-w$. From Remark \ref{rem.ana}, we can rewrite $w=f+g$ satisfying
\[\|f\|_{\hat{K}^{s_p}_{p}(T)}<\infty,\quad \|g\|_{\hat{K}^s_{\tilde{p}}(T)}<\infty.\]
By Hausdroff-Young's inequality, we have $f\in L^{l_1}((t_0,T);L^{\frac{p}{p-1}})$, $\frac{2}{l_1}+\frac{3(p-1)}{p}=1$ for any $0<t_0<T$, $g\in L^{l_2}((0,T);L^{\frac{\tilde{p}}{\tilde{p}-1}})$ and $v\in L^{l_2}((0,T);L^{\frac{\tilde{p}}{\tilde{p}-1}})$ with $\frac{2}{l_2}+\frac{3(\tilde{p}-1)}{ \tilde{p}}=1$. Then, by a usual mollification procedure used in Proposition \text{14.3} in \cite{LRbook}, we have that for any $t_0\leq t\leq T$,
\begin{equation}\label{eq.diff}
\begin{split}
&\int_{\RR^3} w(t,x)\cdot \tilde{w}(t,x)\,\mathrm{d}x+2 \nu\int^t_{t_0}\int_{\RR^3}\nabla w:\nabla \tilde{w}\,\mathrm{d}x\mathrm{d}\tau\\
=&-\int^t_{t_0}\int_{\RR^3}(\delta w \cdot\nabla)\tilde{w}\cdot f\,\mathrm{d}x\mathrm{d}\tau+\int^t_{t_0}\int_{\RR^3}(\delta w\cdot\nabla)\tilde{w}\cdot g\,\mathrm{d}x\mathrm{d}\tau\\
&+\int^t_{t_0}\int_{\RR^3} \big((w\cdot\nabla)\tilde{w}\cdot v+(\tilde{w}\cdot\nabla)w\cdot v\big)\,\mathrm{d}x\mathrm{d}\tau+\int_{\RR^3}w(\delta,x)\cdot \tilde{w}(\delta,x)\,\mathrm{d}x\\
 \triangleq& I_1+I_2+I_3+I_4.
\end{split}
\end{equation}
Due to $g\in L^{l_2}((0,T);L^{\frac{\tilde{p}}{\tilde{p}-1}})$ and $v\in L^{l_2}((0,T);\,L^{\frac{\tilde{p}}{\tilde{p}-1}})$, $\frac{2}{ l_2}+\frac{3(\tilde{p}-1)}{\tilde{p}}=1$, it is obvious that the limits of $I_2$ and $I_3$ exist as $t_0\to 0$.

By Sobolev embedding \eqref{sobolev} and interpolation inequality \eqref{intero}, we have
\begin{align*}
I_1\leq \int^t_{t_0}\left(\|\nabla \delta w\|^2_{L^2}+\|\nabla \tilde{w}\|^2_{L^2}\right)\,\mathrm{d}\tau+C\|f\|_{\hat{K}^s_{p}(T)}^{{2p\over 3-2p}}\int^t_{\delta}\tau^{-1}\|\delta w\|^2_{L^2}\,\mathrm{d}\tau,
\end{align*}
According to Proposition \ref{pro.sma}, we get that for any $0<\delta\ll1$, there exists a $0<\gamma<{{s-s_{\tilde{p}}\over 2\theta}}$, such that for any $0<t<\delta$, we have that
\begin{align*}
\|\delta w(t)\|^2_{L^2}\leq& C\left(\|w(t)-e^{\nu\Delta t}w_0\|^2_{L^2}+\|\tilde{w}(t)-e^{\nu\Delta t}w_0\|^2_{L^2}\right)\\
\leq& C(\|w_0\|_{L^2\cap(F\dot{B}^{s_p}_{p,q}+F\dot{B}^{s}_{\tilde{p},\tilde{q}})},
\|v\|_{Y^s_{\tilde{p},\tilde{q}}(T)})\bigg(t^{\gamma(1-\theta)}
+t^{\frac{s-s_{\tilde{p}}-\gamma \theta}{2}}\bigg).
\end{align*}
Hence,
\begin{align*}
\int^{t_0}_0\tau^{-1}\|\delta w\|^2_{L^2}\,\mathrm{d}\tau\leq & C\big(\|w_0\|_{L^2\cap(F\dot{B}^{s_p}_{p,q}+F\dot{B}^{s}_{\tilde{p},\tilde{q}})},\|v\|_{Y^s_{\tilde{p},\tilde{q}}(T)}\big) \\&\times\int^{t_0}_0\left(\tau^{-1+\gamma(1-\theta)}+\tau^{-1+\frac{s-s_{\tilde{p}}-\gamma \theta}{ 2}}\right)\,\mathrm{d}\tau<\infty,
\end{align*}
which implies the existence of the limit of $I_1$ as $t_0\to 0+$.

Taking $\delta\to 0$ in \eqref{eq.diff}, we  readily get that for any $0\leq t\leq T$,
\begin{equation*}
\begin{split}
&\int_{\RR^3} w(t,x)\cdot \tilde{w}(t,x)\,\mathrm{d}x+2\nu\int^t_0\int_{\RR^3}\nabla w:\nabla \tilde{w}\,\mathrm{d}x\mathrm{d}\tau\\
=&2\|w_0\|^2_{L^2}-\int^t_0\int_{\RR^3}\big((\delta w \cdot\nabla)\tilde{w}\cdot (f+g)+(w\cdot\nabla)\tilde{w}\cdot v+(\tilde{w}\cdot\nabla)w\cdot v\big)\,\mathrm{d}x\mathrm{d}\tau.
\end{split}
\end{equation*}
Since $\tilde{w}$ and $w$ both satisfy the global energy inequality \eqref{glo.W}, from the above equality, we deduce that for any $0\leq t\leq T$,
\begin{align*}
&\|\delta w(t)\|^2_{L^2}+2\nu\int^t_0\|\nabla \delta w\|^2_{L^2}\,\mathrm{d}\tau\\
\leq& 2\int^t_0\int_{\RR^3}(\delta w\cdot\nabla)\delta w\cdot (f+g)\,\mathrm{d}x\mathrm{d}\tau+2\int^t_0\int_{\RR^3} (\delta w\cdot\nabla)\delta w\cdot v\,\mathrm{d}x\mathrm{d}\tau.
\end{align*}
By \eqref{sobolev} and \eqref{intero}, we have for any $0\leq t< \delta$,
\begin{align*}
&\|\delta w\|^2_{L^{\infty}((0,t);\,L^2)}+\nu\|\nabla \delta w\|^2_{L^2(\RR^3\times (0,t))}\\
\lesssim_{\nu}&\Big (\|f\|_{\hat{K}^{s_p}_{p}(T)}^{\frac{2p}{ 3-2p}}+\big(\|g\|_{\hat{K}^{s}_{\tilde{p}}(T)}+\|v\|_{\hat{K}^{s}_{\tilde{p}}(T)}\big)^{\frac{2\tilde{p}}{ 3-2\tilde{p}}}\Big)\|\delta w\|^2_{L^{\infty}((0,t);\,L^2)}\\
&\times\int^t_{0}\big(\tau^{-1+2\gamma(1-\theta)}+\tau^{-1+s-s_{\tilde{p}}}+\tau^{-{s\over s_{\tilde{p}}}}\big)\,\mathrm{d}\tau\\
\lesssim_{\nu}&\Big (\|f\|_{\hat{K}^{s_p}_{p}(T)}^{\frac{2p}{ 3-2p}}+\big(\|g\|_{\hat{K}^{s}_{\tilde{p}}(T)}+\|v\|_{\hat{K}^{s}_{\tilde{p}}(T)}\big)^{\frac{2\tilde{p}}{ 3-2\tilde{p}}}\Big)\|\delta w\|^2_{L^{\infty}((0,t);\,L^2)}t^{\beta},
\end{align*}
where $\beta=\min\big\{2\gamma(1-\theta),s-s_{\tilde{p}}-\gamma \theta,1-{s\over s_{\tilde{p}}}\big\}$.

From it, we can choose a $0<t_1<\delta$ small enough such that $\delta w=0$ on $(0,t_1)\times\RR^3$.  Thanks to $w\in L^{l_1}([t_1,T];\,L^{\frac{p}{ p-1}})+L^{l_2}([t_1,T];\,L^{\frac{\tilde{p}}{\tilde{p}-1}})$, by the continuity argument, we eventually obtain  $w\equiv\tilde{w}$ on $\big(0,T\big)\times\RR^3$.
\end{proof}
\begin{remark}\label{rem.uloc}
	Let us point out that all the above results in this section can be generalized to the $L^2_{\rm uloc}(\RR^3)$ framework by borrowing the idea used in \cite{LRbook}. Here and what in follows,
\[L^2_{\rm uloc}(\RR^3)=\left\{f\in\mathcal{D}'(\RR^3)\big|\,\,\sup_{x\in\RR^3}\|f\|_{L^2(B_1(x))}<\infty \,\,\text{and}\,\,\lim_{|x|\to\infty}\|f\|_{L^2(B_1(x))}=0\right\}.\]
\end{remark}

\section{Proof of Theorem \ref{Thm}}\label{sec.5}
\setcounter{section}{5}\setcounter{equation}{0}
In this section,  we will give a complete proof of Theorem \ref{Thm} by using some results established in previous sections. We will split it into three cases to discuss. \medskip

\noindent\textbf{Case 1:} $1<p<{2/3}$ and $1\leq q<\infty$.

Since $u_0\in F\dot{B}^{s_p}_{p,q}(\RR^3)$, we know by Theorem \ref{well-cri} that there exists a unique local-in-time solution $u\in C([0,T^*);\,F\dot{B}^{s_p}_{p,q}(\RR^3))$ satisfying
\[ \|u\|_{\widetilde{L}^r_{\rm loc}\big([0,T^*);\,F\dot{B}^{\frac2r+s_p}_{p,q}(\RR^3)\big)}<\infty \quad\forall r\in [1,\infty].\]
Moreover, we have by Remark \ref{rem.ana} that
\begin{equation}\label{K}
u\in \hat{K}^{s_p}_p(T)\quad \text{for each}\,\,\,T\in(0,T^*).
\end{equation}
By Lemma \ref{lem.decom}, for each $j\in \ZZ$, there exists $C>0$ such that
\[u_0=v_0+w_0,\]
where $v_0$ and $w_0$ satisfy   $\Div v_0=\Div w_0=0,$
\[\|v_0\|_{F\dot{B}^{s}_{\tilde{p},\tilde{q}}}\leq C2^{-j\theta}\|u_0\|_{F\dot{B}^{s_p}_{p,q}}\quad\text{ and }\quad\|w_0\|_{L^2}\leq C2^{j(1-\theta)}\|u_0\|_{F\dot{B}^{s_p}_{p,q}} ,\]
where $\theta,\,s_p,\, s,\, p,\, q,\, \tilde{p},\,\tilde{q}$ satisfy the restriction condition \eqref{RC}.

By Theorem \ref{well-sub}, we know that the following system admits a unique local mild solution $v\in C([0,T'];\,F\dot{B}^{s }_{\tilde{p},\tilde{q}}(\RR^3))\cap\widetilde{L}^r \big([0,T'];\,F\dot{B}^{2+s }_{\tilde{p},\tilde{q}}(\RR^3)\big)$ with $r\in[1,\infty] $
\begin{equation*}
\left\{\begin{array}{ll}
\partial_t v-\nu\Delta v+v\cdot\nabla v+\nabla Q=0,\\
\Div v=0,\\
v(x,0)=v_0(x).
\end{array}\right.
\end{equation*}
 Since $T'$ only depends on $\|v_0\|_{F\dot{B}^s_{\tilde{p},\tilde{q}}}$, we can choose a suitable $j$ such that $T'>T^*.$  Using Proposition \ref{pro.ana} and Remark~\ref{rem.ana} again, one has
\begin{equation}\label{K'}
v\in \hat{K}^{s}_{\tilde{p}}(T') .
\end{equation}
Letting $w=u-v$, we get easily that $w\in\hat{K}^{s_p}_p(T)+\hat{K}^{s}_{\tilde{p}}(T)$  for each $T\in(0,T^*)$ satisfying the following integral equations:
\begin{equation}\label{eq.Integral}
\begin{split}
w=&e^{\nu t\Delta}w_0+\int^t_0 e^{\nu(t-\tau)\Delta}\mathbb{P}\big(w\cdot\nabla w\big)\,\mathrm{d}\tau+\int^t_0 e^{\nu(t-\tau)\Delta}\mathbb{P}\big(v\cdot\nabla w+w\cdot\nabla v\big)\,\mathrm{d}\tau\\
\triangleq& y_w+B(w,w)+L(w).
\end{split}
\end{equation}
Our task is now to show that
\begin{equation}
w\in L^{\infty}((0,T^*);L^2)\cap L^2((0,T^*);\dot{H}^1).
\end{equation}
Since $w_0\in L^2(\RR^3)\cap (F\dot{B}^{s_p}_{p,q}(\RR^3)+F\dot{B}^s_{\tilde{p},\tilde{q}}(\RR^s))$, we have by Banach fixed point theorem that there exists a local solution $\bar{w}\in X_{T_0}$, $T_0<T^*$ of  \eqref{eq.Integral}, where
$$X_{T_0}\triangleq \big(\hat{K}^{s_p}_p(T_0)+\hat{K}^{s}_{\tilde{p}}(T_0)\big)\cap E(T_0)$$
with $E(T_0)\triangleq C_{\rm b}([0,T_0);\,L^2)\cap L^2((0,T_0);\,\dot{H}^1)$.
 And $ \bar{ w}$ solves the following problem on $(0,T_0)$
\begin{equation}\label{W-2.sec5}
\left\{\begin{array}{ll}
\partial_t w-\nu\Delta w+w\cdot\nabla w+v\cdot\nabla w+w\cdot\nabla v+\nabla \bar{P}=0,\\
\Div w=0,\\
w(x,0)=w_0.
\end{array}\right.
\end{equation}
Indeed, by Lemma \ref{norm} and H\"older's inequality, we have
\[\|y_w\|_{X_{T_0}}\leq C(\|w_0\|_{F\dot{B}^{s_p}_{p,q}+F\dot{B}^{s}_{\tilde{p},\tilde{q}}}+\|w_0\|_{L^2}).\]
On the other hand, by a simple calculation and  estimate  \eqref{eq.ana-1} in Proposition \ref{pro.ana}, we have that the following estimates for any $w=w_1+w_2$ and $w'=w'_1+w'_2$
\begin{equation*}
\left\{\begin{aligned}
&\|B(w_1,w'_1)\|_{\hat{K}^{s_p}_p(T_0)}\lesssim_{\nu}\|w_1\|_{\hat{K}^{s_p}_p(T_0)}
\|w_1'\|_{\hat{K}^{s_p}_{p}(T_0)},\\
&\|B(w_1,w'_2)\|_{\hat{K}^{s_p}_p(T_0)}\lesssim_{\nu}T_0^{\frac{s-s_{\tilde{p}}}{2}}\|w_1\|_{\hat{K}^{s_p}_p(T_0)}
\|w_2'\|_{\hat{K}^{s}_{\tilde{p}}(T_0)},\\
&\|B(w_2,w'_1)\|_{\hat{K}^{s_p}_p(T_0)}\lesssim_{\nu}T_0^{\frac{s-s_{\tilde{p}}}{2}}
\|w_2\|_{\hat{K}^{s_p}_{p}(T_0)}\|w_1'\|_{\hat{K}^{s}_{\tilde{p}}(T_0)},\\
&\|B(w_2,w'_2)\|_{\hat{K}^{s}_{\tilde{p}}(T_0)}\lesssim_{\nu} T_0^{\frac{s-s_{\tilde{p}}}{2}}\|w_2\|_{\hat{K}^{s}_{\tilde{p}}(T_0)}
\|w_2'\|_{\hat{K}^{s}_{\tilde{p}}(T_0)},\\
&\|B(w,w)\|_{E(T_0)}\lesssim_{\nu} \Big(\|w_1\|_{\hat{K}^{s_p}_p(T_0)}+T_0^{\frac{s-s_{\tilde{p}}}{2}}\|w_2\|_{\hat{K}^s_{\tilde{p}}(T)}
\Big)\|w\|_{E(T_0)},
\end{aligned}\right.
\end{equation*}
and
\begin{equation*}
\left\{\begin{aligned}
&\|L(w_1)\|_{\hat{K}^{s_p}_p(T_0)}\lesssim_{\nu}T_0^{\frac{s-s_{\tilde{p}}}{2}}\|w_1\|_{\hat{K}^{s_p}_p(T_0)}
\|v\|_{\hat{K}^{s}_{\tilde{p}}(T_0)},\\
&\|L(w_2)\|_{\hat{K}^{s}_{\tilde{p}}(T_0)}\lesssim_{\nu}T_0^{\frac{s-s_{\tilde{p}}}{2}}\|w_2\|_{\hat{K}^{s}_{\tilde{p}}(T_0)}
\|v\|_{\hat{K}^{s}_{\tilde{p}}(T_0)},\\
&\|L(w)\|_{E(T_0)}
\lesssim_{\nu} T_0^{\frac{s-s_{\tilde{p}}}{2}}\big(\|v\|_{\hat{K}^s_{\tilde{p}}(T_0)}
+\|\nabla v\|_{\hat{K}^{s-1}_{\tilde{p}}(T_0)}\big)\|w\|_{E(T_0)}.
\end{aligned}\right.
\end{equation*}

Collecting all above estimates, by Lemma \ref{fixed} and the continuity argument, there exists a local solution $\bar{w}\in X_{T_0}$, $T_0<T^*$ of equations \eqref{eq.Integral}.

According to \eqref{W-2.sec5},  we get by performing $L^2$-energy estimate that
 \begin{equation*}
\sup_{0\leq t\leq T_0}\|\overline{w}(t)\|^2_{L^2}+\nu\int^{T_0}_0\|\nabla \overline{w}\|^2_{L^2}\,\mathrm{d}\tau\leq \|w_0\|^2_{L^2}e^{C(\nu){T^*}^{1-{s\over s_{\tilde{p}}}}\|v\|^{2\tilde{p}\over 3-2\tilde{p}}_{\hat{K}^s_{\tilde{p}}(T^*)}}.
\end{equation*}
Repeating the above same process in finite times, we finally obtain that equations \eqref{eq.Integral}  admit a local mild solution $\bar{w}\in C([0,T^*);\,L^2)$ satisfying
\[\bar{w}\in X_{T }\quad\text{for each}\quad T\in (0,T^*)\]
and uniform estimate
\begin{equation}\label{est.unif}
\sup_{0\leq t<T^*}\|\bar{w}(t)\|^2_{L^2}+\nu\int^{T^*}_0\|\nabla \bar{w}\|^2_{L^2}\,\mathrm{d}\tau\leq \|w_0\|^2_{L^2}e^{C(\nu){T^*}^{1-{s\over s_{\tilde{p}}}}\|v\|^{2\tilde{p}\over 3-2\tilde{p}}_{\hat{K}^s_{\tilde{p}}(T^*)}}.
\end{equation}
By uniqueness theorem, we know that $\bar{w}(x,t)\equiv w(x,t)$ for $t\in[0,T^*).$
Thus we can say $w$ satisfies the global energy inequality \eqref{glo.W} for any $0<t\leq T^*$.

We turn to deal with  the pressure $\bar{P}$.  From equations \eqref{W-2}, it follows  that
\[-\Delta \bar{P}=\Div\Div(w\otimes w+w\otimes v+v\otimes w) \quad\text{in}\quad \RR^3\times (0,T^*).\]
Therefore,
\[\bar{P}=\partial_i\partial_jK\ast(w_i w_j)-\partial_i\partial_jK\ast(w_i v_j+v_i w_j)\triangleq \bar{P}_1+\bar{P}_2\]
with $K(x)$ is the kernel function of $(-\Delta)^{-1}$.
Since $v\in \hat{K}^s_{\tilde{p}}(T^*)$ with $s>s_{\tilde{p}}$, by Hausdorff-Young's inequality, we have $v\in L^{l}((0,T^*);L^{\tilde{p}'})$ satisfying $\frac 2l+\frac3{p'}=1$ where $\tilde{p}'$ is the conjugate index of $\tilde{p}$. By Calder\'{o}n-Zygmund estimates, H\"{o}lder's inequality and \eqref{sobolev}, we have,
\begin{equation}\label{est.q1}
\|\bar{P}_1\|_{L^{{3\over 2}}(\RR^3\times (0,T^*))}\leq C\|w\|^2_{L^3(\RR^3\times(0,T^*))}\leq C(T^*)^{1/6}\|w\|_{L^{\infty}((0,T^*);\,L^2)}\|\nabla w\|_{L^2((0,T^*);\,L^2)}
\end{equation}
and
\begin{equation}\label{est.q2}
\begin{split}
\|\bar{P}_2\|_{L^2(\RR^3\times (0,T^*))}\leq C\|w\|^{{2/ r}}_{L^{\infty}((0,T^*);\,L^2)}\|w\|^{1-{2/r}}_{L^2((0,T^*);\,\dot{H}^1)}
\|v\|_{L^r((0,T^*);\,L^{{\tilde{p}'}})},
\end{split}
\end{equation}
which yields $\bar{P}\in L^{{3/2}}(\RR^3\times(0,T^*))+L^2(\RR^3\times (0,T^*))$.
Hence, it's obvious that $w$ satisfies the local energy inequality \eqref{loc.W} for any $\varphi\in \mathcal{D}(\RR^3\times(0,T^*))$.  In conclusion, we have proved that $w$ is a energy solution of satisfying $(\text{W}1)$-$(\text{W}5).$

Next, we show the existence of singular points of local mild solution to \eqref{NS} with initial data in $F\dot{B}^{s_p}_{p,q}(\RR^3)$ with $1<p<{3\over 2}$, $1\leq q<\infty$.
\begin{proposition}\label{pro.sin}
Let $u\in C([0,T^*);F\dot{B}^{2-{3\over p}}_{p,q}(\RR^3))$ be the mild solution to \eqref{NS} with $1<p<{3\over 2}$, $1\leq q<\infty$, and $T^*$ denotes the maximal existence time. Suppose $T^*<\infty$, then $u$ has a singular point at $T^*$, that is, there exists $z_0=(x_0,T^*)$, such that $\|u\|_{L^{\infty}(Q_r(z_0))}=+\infty$ for any $r>0$.
\end{proposition}
\begin{proof}
According to Remark \ref{rem.ana}, we know that there exists a $0<t_0<T^*$ such that $u(t_0)\in FL^p(\RR^3)$. Then, we decompose $u(t_0)$ as follows:
\[u(t_0)=\mathcal{F}(\hat{u}(t_0)I_{\{|\hat{u}(t_0)|> \lambda\}})+\mathcal{F}(\hat{u}(t_0)I_{\{|\hat{u}(t_0)|\leq \lambda\}})\triangleq v_0+w_0,\]
where $\lambda$ is a real number to be fixed later.

By a simple calculation, one has
  $v_0\in FL^p(\RR^3)$, $w_0\in L^2(\RR^3)$.
Thanks to \eqref{orth}, there exist $C_1>0$ and $C_2>0$ such that
$$C_1\|v_0\|_{F\dot{B}^0_{p,p}}\leq\|v_0\|_{FL^p}\leq C_2\|v_0\|_{F\dot{B}^0_{p,p}}.$$
Therefore, by Theorem \ref{well-sub}, there
exists a mild solution $v\in C([t_0,T'],FL^p)$ of system \eqref{NS} corresponding  to the initial
data $v(t_0)=v_0$. From Proposition \ref{pro.ana} and Remark \ref{rem.ana}, we have
\[\sup_{0\leq t< T }t^{\frac \alpha 2+\frac{3}{2\eta}-\frac 3{2p}}\|\nabla^{\alpha}
   v(t)\|_{FL^{\eta}}<\infty,\quad \forall (\alpha,\eta)\in \{0,1\}\times [1,p].\]
Since $T'$ depend on $\|v_0\|_{FL^p}$,   we choose $\lambda$ large enough such that $\|v_0\|_{FL^p}$ is small and then $T'>T^*$. Thus, we can deduce that for any $0<\delta<T^*$,
\begin{align*}
&\|v\|^q_{\tilde{L}^{2}((T^*-\delta,T^*);\,F\dot{B}^{s_p+1}_{p,q})}\\
=&\sum_{j\leq 0}2^{jq(3-{3\over p})}\|\dot{\Delta}_j v\|^q_{L^2((T^*-\delta,T^*);\,FL^p)}+\sum_{j> 0}2^{jq(3-{3\over p})}\|\dot{\Delta}_j v\|^q_{L^2((T^*-\delta,T^*);\,FL^p)}\\
\lesssim& \sum_{j\leq 0} 2^{jq(3-{3\over p})}\delta^{{q\over 2}}\|v\|^q_{L^{\infty}((T^*-\delta,T^*);\,FL^p)}+\sum_{j> 0}2^{jq(2-{3\over p})}\delta^{{q\over 2}}\|\nabla v\|^q_{L^{\infty}((T^*-\delta,T^*);\,FL^p)}<\infty.
\end{align*}
Let $w=u-v$ and $\bar{P}$ is the associated pressure. In the same way as used in the first part of this section, we know $w$ is a local energy solution of the following equations \begin{equation}\label{W-W2}
\left\{\begin{array}{ll}
\partial_t w-\nu\Delta w+w\cdot\nabla w+v\cdot\nabla w+w\cdot\nabla v+\nabla \bar{P}=0,\quad(x,t)\in\RR^3\times(t_0,T^*),\\
\Div w=0,\\
 w(x,t_0)=w(t_0),
\end{array}\right.
\end{equation}
where $v\in L^{\infty}((0,T^{'});\,FL^p)$  , and $w,\,\bar{P}$ satisfy the following estimates:
\begin{align*}
&\sup_{t_0\leq t<T^*}\|w\|_{L^2(\RR^3)}+\|\nabla w\|_{L^2(\RR^3\times(t_0,T^*))}+\|\bar{P}\|_{(L^2+L^{{3\over 2}})(\RR^3\times (t_0,T^*))}\\
\leq& C(\nu,\|w_0\|_{L^2(\RR^3)},T^*,\|v\|_{L^{\infty}([t_0,T^*];\,FL^p(\RR^3))}).
\end{align*}
By the same argument used in \cite[Proposition 33.2]{LRbook}, we can show that
\begin{equation*}
\begin{split}
&\|w(t)\|^2_{L^2_{\rm uloc}(\RR^3\setminus B_{2R}(0))}+\int_0^t\|\nabla w(t')\|^2_{L^2_{\rm uloc}(\RR^3\setminus B_{2R}(0))}\,\mathrm{d}t'\\
\leq& C(\nu,\|w_0\|_{L^2(\RR^3)},T^*,\|v\|_{L^{\infty}([t_0,T^*];\,FL^p(\RR^3))})\left(\|w_0\|^2_{L^2_{\rm uloc}(\RR^3\setminus B_{2R}(0))}+\frac{1+\log R}{R}\right).
\end{split}
\end{equation*}
So, we conclude that there exists a $R>0$, such that for any $|x_0|>R$,
\[R_0^{-2}\int_{Q_{R_0}(x_0,T^*)}\big(|w|^3+|\bar{P}|^{{3\over 2}}\big)\,\mathrm{d}x\mathrm{d}t<\epsilon_0,\]
where $R_0=\sqrt{2\delta}$ and $\epsilon_0$ is the constant in Theorem \ref{regularity}. Since $(w,\bar{P})$ is a suitable weak solution to equations \eqref{W-2} with $v\in L^{\infty}_tL^{{p\over p-1}}_x(Q_{{R_0}}(x_0,T^*))$ in $Q_{R_0}(x_0,T^*)$ and by $\varepsilon$-regularity criterion in Theorem \ref{regularity}, we have $w$ is bounded in $(\RR^3\setminus B_R(0))\times [{T^*-\delta},T^*]$.

Suppose Proposition \ref{pro.sin} is false. Then $u$ is bounded in a neighborhood of any point in $\RR^3\times [T^*-\delta, T^*]$. Since $v\in L^{\infty}(\RR^3\times [{T^*-\delta}, T^*])$. Thus, $w$ is bounded in a neighborhood of any point in $\RR^3\times [{T^*-\delta}, T^*]$. By Vitali covering theorem, we know that $w$ is bounded in $B_R(0)\times [{T^*-\delta}, T^*]$. This shows that $w\in L^{\infty}(\RR^3\times [{T^*-\delta}, T^*])$. Combining this fact with $w\in L^{\infty}((0,T^*);L^2)\cap L^2((0,T^*);\dot{H}^1)$ induces us to claim that $$w\in \tilde{L}^2((T^*-\delta,T^*);F\dot{B}^{s_p+1}_{p,q}(\RR^3)).$$

Setting $\displaystyle w=\sum_{j\leq 0}\dot{\Delta}_jw+\sum_{j>0}\dot{\Delta}_j w\triangleq w^l+w^h$, we can easily get that
\begin{align*}
\|w^l\|_{\tilde{L}^2((T^*-\delta,T^*);\,F\dot{B}^{1+s_p}_{p,q})}\lesssim \|w^l\|_{\tilde{L}^2((T^*-\delta,T^*);\,F\dot{B}^{{3\over 2}}_{2,q})}\lesssim \delta^{{1\over 2}}\|w\|_{L^{\infty}([t_0,T^*];\,L^2)}.
\end{align*}
Thus, we only need to prove $w^h\in \tilde{L}^2((T^*-\delta,T^*);F\dot{B}^{s_p+1}_{p,q}(\RR^3))$.

We rewrites
\begin{align*}
w(x,t)=&e^{\nu (t-T^*+\delta)\Delta}w(T^*-\delta)+\int^t_{T^*-\delta} e^{\nu(t-\tau)\Delta}\mathbb{P}(w\cdot\nabla w)\,\mathrm{d}\tau\\
&+\int^t_{T^*-\delta} e^{\nu(t-\tau)\Delta}\mathbb{P}(w\cdot\nabla v+v\cdot\nabla w)\,\mathrm{d}\tau\\
\triangleq& G+B(w,w)+L(w),
\end{align*}
and $$w^h(x,t)=\sum_{j>0}\dot{\Delta}_j G+\sum_{j>0}\dot{\Delta}_j B(w,w)+\sum_{j>0}\dot{\Delta}_j L(w)\triangleq G^h+B^h(w,w)+L^h(w).$$
For $G^h$, by \eqref{est.TS}, we get
\[\|G^h\|_{\widetilde{L}^2((T^*-\delta,T^*);\,F\dot{B}^{s_p+1}_{p,q})}\leq \bigg(\sum_{j\geq -1}2^{jq(2-{3\over p})}\bigg)^{{1\over q}}\|w(T^*-\delta)\|_{FL^p}<\infty.\]
For
$B^h(w,w)$, we have
\begin{align*}
&\|B^h(w,w)\|_{\tilde{L}^2((T^*-\delta,T^*);\,F\dot{B}^{s_p+1}_{p,q})}\\
\lesssim& \sum_{j\geq -1}2^{{3\over 2}jq}\|\dot{\Delta}_jB(w,w)\|^q_{L^2((T^*-\delta,T^*);\,L^2)}\\
\lesssim &\sum_{j\geq -1}2^{{3\over 2}jq}\Big\|\int^t_{T^*-\delta}e^{-\nu(t-\tau)2^{2j}}\|w(\tau)\|_{L^{\infty}}\| \nabla w(\tau)\|_{L^2}\,\mathrm{d}\tau\Big\|^q_{L^2((T^*-\delta,T^*)}\\
\lesssim& \sum_{j\geq -1}2^{-{1\over 2}jq}\|w\|^q_{L^{\infty}(\RR^3\times(T^*-\delta,T^*))}\| \nabla w\|^q_{L^2((T^*-\delta,T^*);\,L^2)}<\infty.
\end{align*}
Similarly, we have by the H\"older inequality that
\begin{align*}
 \|L^h(w)\|_{\tilde{L}^2((T^*-\delta,T^*);\,F\dot{B}^{s_p+1}_{p,q})}
\lesssim &\delta^{{q\over 2}}\|\nabla v\|^q_{L^{\infty}(\RR^3\times(T^*-\delta,T^*))}\| w\|^q_{L^{\infty}((T^*-\delta,T^*);\,L^2)}\\&+\|v\|^q_{L^{\infty}(\RR^3\times(T^*-\delta,T^*))}\| \nabla w\|^q_{L^2((T^*-\delta,T^*);\,L^2)}<\infty.
\end{align*}
Collecting above estimates, we get that $w\in \tilde{L}^2((T^*-\delta,T^*);F\dot{B}^{s_p+1}_{p,q}(\RR^3))$. This, together with $v\in \tilde{L}^2((T^*-\delta,T^*);F\dot{B}^{s_p+1}_{p,q}(\RR^3))$, gives $u\in \tilde{L}^2((T^*-\delta,T^*);F\dot{B}^{s_p+1}_{p,q}(\RR^3))$. According the blow-up criterion in Theorem \ref{blow}, $u$ can be extend beyond to $t=T^*$, which gives a contradiction. Then Proposition \ref{pro.sin} is proved.
\end{proof}
Now we come back to the proof of Theorem \ref{Thm}.

Since $\rho_{\max}<\infty$, we can choose a sequence $u^{(k)}_0\in F\dot{B}^{s_p}_{p,q}(\RR^3)$ such that $T^*(u^{(k)}_0)<\infty$ and $\|u^{(k)}_0\|_{F\dot{B}^{s_p}_{p,q}}\searrow \rho_{\max}$. According to Proposition \ref{pro.sin}, the mild solution $u^{(k)}$ associated to initial data $u^{(k)}_0$ has a singular point $(x^{(k)},T^*(u^{(k)}_0))$. After the following scaling and translation
\[u^{(k)}_0(x)\to 2^{\lambda_k} u^{(k)}_0(2^{\lambda_k}(x-x^{(k)}))\quad\text{and}\quad
u^{(k)}(x,t)\to 2^{\lambda_k} u^{(k)}(2^{\lambda_k}(x-x^{(k)}),2^{2\lambda_k} t)\]
with $\lambda_k\in\ZZ$, the singular point $(x^k,T^*(u^{(k)}_0))$ becomes $(0,2^{-2\lambda_k}T^*(u^{(k)}_0))$. We still denote the sequence after the above translation and scaling as $u^{(k)}$, so does $u^{(k)}_0$. Then we can choose a series $\lambda_k$ such that $T^*_k \triangleq 2^{-2\lambda_k}T^*(u^{(k)}_0)\in (\frac 12,1)$ for any $k\in\NN$. Thus, there exists a subsequence of $\lambda_k$, still denoted by $\lambda_k$, such that $T^*_k$ converge to a point $t^*$. Still denote the corresponding subsequence of $u^{(k)}$ by $u^{(k)}$, so does $u^{(k)}_0$. Thus, we can conclude that the mild solution $u^{(k)}$ associated to initial data $u^{(k)}_0$ has a singular point $(0,T^*_k)$ with $T^*_k\in(\frac 12,1)$ and $T^*_k\to t^*$ as $k\to\infty$.

According to $\|u^{(k)}_0\|_{F\dot{B}^{s_p}_{p,q}}\searrow \rho_{\max}$, we can easily get $\|u^{(k)}_0\|_{F\dot{B}^{s_p}_{p,q}}\leq M$ for some constant $M>0$. According to Lemma \ref{lem.decom}, for any $j\in\ZZ$, we can split $u^{(k)}_0$ into
 $v^{(k)}_0+w^{(k)}_0$ with $\Div v^{(k)}_0=\Div w^{(k)}_0=0$,
\begin{equation}\label{u.v_0}
\|v^{(k)}_0\|_{F\dot{B}^{s}_{\tilde{p},\tilde{q}}}\leq C2^{-j\theta}\|u^{(k)}_0\|_{F\dot{B}^{s_p}_{p,q}}\leq C2^{-j\theta}M
\end{equation}
and
\begin{equation}\label{u.w_0}
\|w^{(k)}_0\|_{L^2}\leq C2^{j(1-\theta)}\|u_0\|_{F\dot{B}^{s_p}_{p,q}}\leq C2^{j(1-\theta)}M,
\end{equation}
where $s_p,\,s,\,p,\,q,\,\tilde{p},\,\tilde{q}$ satisfy \eqref{RC}.
Moreover, for suitable larger $j$, $u^{(k)}$ can be decomposed as $u^{(k)}=v^{(k)}+w^{(k)}$, where $v^{(k)}$ is a unique local mild solution to \eqref{NS} established in Theorem~\ref{well-sub} stating from $v^{(k)}_0$ on $\RR^3\times (0,1)$ and $(w^{(k)},\bar{P}^{(k)})$ is the   energy weak solution of equations \eqref{W-2} on $\RR^3\times (0,T^*_k)$ with $v$ replaced by $v^{(k)}$ associated to initial data $w^{(k)}_0$. Due to Remark \ref{rem.ana}, $v^{(k)}$ satisfies for any $(\alpha,\eta)\in \{0,1\}\times [1,p]$,
\begin{equation}\label{eq.V1}
\sup_{0\leq t\leq 1 }t^{\frac \alpha 2-\frac s2+\frac{3}{2\eta}-\frac 3{2\tilde{p}}}\|\nabla^{\alpha} v^{(k)}\|_{FL^{\eta}}\leq C(M,j).
\end{equation}
Form \eqref{eq.V1}, we can easily get that for any multiindex $\alpha$ satisfying $|\alpha|=0,1,2$, $D^{\alpha} v^{(k)}$ is uniformly bounded in $L^{\infty}([\delta,1];L^{\infty}(\RR^3))$ for any $0<\delta <1$ and $v^{(k)}$ is uniformly bounded in $L^{r}([0,1];\,L^{{\tilde{p}\over {\tilde{p}}-1}}(\RR^3)))^2$ with ${2\over r}+{3(\tilde{p}-1)\over \tilde{p}}=1$. This, together with the facts that $v^{(k)}$ satisfies
\[\partial_t v^{(k)}-\nu\Delta v^{(k)}+\mathbb{P}(v^{(k)}\cdot\nabla v^{(k)})=0\]
on $(0,1)\times \RR^3$, entails that
$$\|\partial_t v^{(k)}\|_{L^{\infty}([\delta,1];\,L^{\infty})}\leq C(M,j).$$
By Ascoli-Arzela theorem, there exist a subsequence of $v^{(k)}$, still denoted by $v^{(k)}$, and $v$ such that
\[v^{(k)}\overset{*}\rightharpoonup v\,\text{ in }\,\hat{K}^{s}_{\tilde{p}}(1),\quad v^{(k)}\rightarrow v\text{ in } C_{\rm loc}(\RR^3\times (0,1]),\quad
v^{(k)}_0\rightharpoonup v_0\,\text{ in }\,F\dot{B}^s_{\tilde{p},\tilde{q}}(\RR^3).\]
In addition, we have
\[\|\nu\Delta v^{(k)}\|_{L^r([0,1];W^{-2,{\tilde{p}\over {\tilde{p}}-1}})}+\|\mathbb{P}(v^{(k)}\cdot\nabla v^{(k)})\|_{L^{{r\over 2}}([0,1];W^{-1,{2\tilde{p}\over \tilde{p}-1}})}\leq C,\]
which implies
\[\int_{\RR^3}v^{(k)}(t)\varphi\,\mathrm{d}x \longrightarrow \int_{\RR^3}v(t)\varphi\,\mathrm{d}x\quad\text{in} \quad C([0,1]),\]
for any $\varphi\in \mathcal{D}(\RR^3)$.
According to the above convergence, we can say that $v\in \hat{K}^{s}_{\tilde{p}}(1)$ is a unique mild solution of   equations \eqref{NS} on $\RR^3\times (0,1)$ with initial data $v_0\in F\dot{B}^s_{\tilde{p},\tilde{q}}(\RR^3)$.

By Proposition \ref{pro.uni'}, there exists a local energy solution $(\tilde{w}^{(k)}, \tilde{P}^{(k)})$ to equations \eqref{W-2} on $\RR^3\times(0,1)$ with initial data $w^{(k)}_0$ associated to $v^{(k)}$, satisfying  $w^{(k)}=\widetilde{w}^{(k)}$ on $[0,T^k]$. By using estimates \eqref{u.w_0} and \eqref{eq.V1}, we can easily get
\[\sup_{0\leq t\leq 1}\|w^{(k)}(t)\|^2_{L^2}+\nu\int^{1}_0\|\nabla w^{(k)}\|^2_{L^2}\,\mathrm{d}\tau+\|\tilde{P}^{(k)}\|_{L^2(\RR^3\times (0,1))+L^{3/2}(\RR^3\times (0,1))}\leq C(M,j).\]

From equations \eqref{W-2}, we can get that $\{\partial_t w^{(k)}\}_{k}$ is uniformly bounded in $L^{3/2}((0,1);H^{-2})$. Moreover, for any $r>0$,
\[H^1_x(B_r)\hookrightarrow\hookrightarrow L^{3/2}_x(B_r)\hookrightarrow H^{-2}_x(B_r)\]
and $H^1_x(B_1)$ is reflexive. Thus, by Aubin-Lions Lemma in Chapter 5 of Seregin \cite{Se15}, we obtain that $\{w^{(k)}\}_k$ is compact in $L^{3/2}(B_r\times (0,1))$ for any $r>0$. By interpolation between $L^{\infty}((0,1);L^2(B_r)$ and $L^2((0,1);\dot{H}^1(B_r))$, we have $\{w^{(k)}\}_k$ is bounded in $L^{10/3}(B_r\times (0,1))$. Hence, we conclude that $w^{(k)}$ is compact in $L^3(B_r\times (0,1))$. Therefore, there exists a subsequence of $(\tilde{w}^{(k)},\tilde{P}^{(k)}, v^{(k)})$, still denoted by $(\tilde{w}^{(k)},\tilde{P}^{(k)}, v^{(k)})$, and a suitable weak solution $(w,\bar{P})$ to equations \eqref{W-2} associated to $v$ such that
\begin{itemize}
  \item in $\tilde{w}^{(k)}\overset{*}\rightharpoonup w$\text{ in }$L^{\infty}([0,1];L^2)$ and $\tilde{w}^{(k)}\rightharpoonup w$ in $L^2((0,1);\dot{H}^1)$
  \item $\tilde{w}^{(k)}\rightharpoonup w_0$ in $L^2$
  \item $\tilde{w}^{(k)}\rightarrow w$ in $L^3_{\rm loc}(\RR^3\times [0,1])$ and $\tilde{P}^{(k)}\rightharpoonup \bar{P}$ in $L^{{3\over 2}}(\RR^3\times [0,1])+L^{2}(\RR^3\times [0,1])$
  \item $w$ is weakly continuous in $L^2$ with respect to $t\in [0,1]$.
\end{itemize}
Whit the above convergence, one can easily deduce that $(w, \bar{P})$ is a suitable weak solution to \eqref{eq.W'} associated to $v$ on $\RR^3\times (0,1)$. By stability of singularity in Proposition \ref{pro.stability}, we can say $w$ has a singular point $(0,t^*)$. In addition, by Proposition \ref{pro.sma}, there exists a $0<\gamma<{s-s_{\tilde{p}}\over 2\theta}$ such that for any $0<t<<1$,
\[\|\tilde{w}^{(k)}-e^{\nu\Delta t}w^{(k)}_0\|_{L^2}\leq C(M)\Big(t^{\gamma(1-\theta)}+t^{\frac{s-s_{\tilde{p}}-\gamma\theta}{2}}\Big).\]
So, we have that for any $0<t<<1$
\[\|w(t)-e^{\nu\Delta t}w_0\|_{L^2}\leq  \liminf_{k\to\infty}\|\tilde{w}^{(k)}(t)-e^{\nu\Delta t}w^{(k)}_0\|_{L^2}\leq C(M)\Big(t^{\gamma(1-\theta)}+t^{\frac{s-s_{\tilde{p}}-\gamma\theta}{ 2}}\Big),\]
which implies that $\displaystyle \lim_{t\to 0+}\|w(t)-w_0\|_{L^2}=0$.
This together with the above convergence enables us to conclude  that $w$ is a local energy solution to equations \eqref{W-2} on $\RR^3\times (0,1)$.

Let $u\in C([0,T^*);F\dot{B}^{s_p}_{p,q})$ and $\tilde{v}\in C([0,T_2);F\dot{B}^{s}_{\tilde{p},\tilde{q}})$, be the unique mild solution established in Theorem \ref{well-cri} or Theorem \ref{well-sub} starting from $u_0$ and $v_0$, respectively. By the uniqueness, we have $v=\tilde{v}$ on $\RR^3\times (0,1)$ and $T_2\geq 1$. Set $\tilde{w}=u-v$, we have by Proposition \ref{pro.uni'} that $\tilde{w}=w$ on $\RR^3\times (0,\min\{T^*,1\})$ which implies $u=v+w$ on $\RR^3\times(0,\min\{T^*,1\})$. Since $w$ is singular at $t^*<1$, one has $T^*\leq t^*<1$. Relying on the definition of $\rho_{\max}$, we have that
$$\rho_{\max}\leq \|u_0\|_{F\dot{B}^{s_p}_{p,q}}\leq \liminf_{k\to\infty}\|u^{(k)}_0\|_{F\dot{B}^{s_p}_{p,q}}\leq \rho_{\max}.$$
Thus $\|u_0\|_{F\dot{B}^{s_p}_{p,q}}=\rho_{\max}$ and $\|u^{(k)}_0\|_{F\dot{B}^{s_p}_{p,q}}\rightarrow \|u_0\|_{F\dot{B}^{s_p}_{p,q}}$. This, together with the fact that $F\dot{B}^{s_p}_{p,q}(\RR^3)$ is a uniformly convex Banach space, gives $u^{(k)}_0\rightarrow u_0$ in $F\dot{B}^{s_p}_{p,q}(\RR^3)$.\medskip

\noindent\textbf{Case 2:} ${2/3}\leq p\leq \infty$ and $2< q<\infty$.

Since $q>2$, by Lemma \ref{lem.inter}, for any $j\in \ZZ$, we can split $u_0\in F\dot{B}^{2-{3\over p}}_{p,q}(\RR^3)$ into $f_0+g_0$ satisfying $\Div f_0=\Div g_0=0$ and
\[\|f_0\|_{F\dot{B}^{2-{3\over p}}_{p,2}}\leq C2^{-j\theta}\|u_0\|_{F\dot{B}^{2-{3\over p}}_{p,q}},\quad \|g_0\|_{F\dot{B}^{2-{3\over p}}_{p,q_1}}\leq C2^{j(1-\theta)}\|u_0\|_{F\dot{B}^{2-{3\over p}}_{p,q}},\]
where
$\frac1q=\frac\theta {q_1}+\frac{1-\theta}{2}=1$ and $\theta\in (0,1)$.

From Theorem \ref{well-cri}, we can choose a $j\in\ZZ$ small enough such that there exists a unique global in time mild solution $g\in C([0,\infty);\,F\dot{B}^{2-{3\over p}}_{p,q_1})\cap \widetilde{L}^{r}([0,\infty);\,F\dot{B}^{2-{3\over p}+{2\over r}}_{p,q_1})$ with $ r\in[1,\infty)$ to \eqref{NS} starting from $g_0$. Thanks to $F\dot{B}^{2-{3\over p}}_{p,2}(\RR^3)\hookrightarrow F\dot{B}^{2-{3\over p}}_{p,q}(\RR^3)\hookrightarrow F\dot{B}^{2-{3\over p}}_{p,q_1}(\RR^3)$ and Lemma \ref{regul}, we can deduce that
$$f\triangleq NS(u_0)-g\in C\Big([0,T^*(u_0));\,F\dot{B}^{2-{3\over p}}_{p,2}\Big)\cap \widetilde{L}^{r}_{\rm loc}\Big([0,T^*(u_0));\,F\dot{B}^{2-{3\over p}+{2\over q}}_{p,2}\Big),\quad \forall  r \in(1,\infty),$$
satisfies the following equations:
\begin{equation}\label{f}
\left\{\begin{array}{ll}
\partial_t f-\nu\Delta f+f\cdot\nabla f+f\cdot\nabla g+g\cdot\nabla f+\nabla Q=0,\\
\Div f=0,\\
f(x,0)=f_0.
\end{array}\right.
\end{equation}
Note that $T^*(u_0)$ is also the maximal existence time of $f$. To prove Theorem \ref{Thm}, we only need to prove the existence of minimal blow-up initial data to problem \eqref{f} in $F\dot{B}^{2-\frac3p}_{p,2}(\RR^3)$ with $p\geq 3/2$.

Following the method developed by Gallagher and Planchon \cite{GP02}, and using Bony-paraproduct decomposition, we  have
\[(a,b,c)\in E\times E\times L^q\Big((0,T);\,F\dot{B}^{2-\frac3p}_{p,q}(\RR^3)\Big)\longmapsto \int^T_0\int_{\RR^3} (a\cdot\nabla b)\cdot c\,\mathrm{d}x\mathrm{d}\tau,\]
 is continuous for any $p\geq 3/2$ and $1\leq q<\infty$. Then, we can translated the existence of minimal
 blow-up initial data to problem \eqref{f} in $F\dot{B}^{2-\frac3p}_{p,2}(\RR^3)$ into the  existence of minimal initial data to \eqref{NS} in $F\dot{B}^{2-\frac3p}_{p,2}(\RR^3)$.

Hence, it is suffices to show Theorem \ref{Thm} in $F\dot{B}^{2-\frac{3}{p}}_{p,2}(\RR^3)$. By embedding theorem, one has
\[F\dot{B}^{2-\frac{3}{p}}_{p,2}(\RR^3)\hookrightarrow L^2_{\rm uloc}(\RR^3).\]
By Remark \ref{rem.uloc} and existence of singularities, we can get the desired result by repeating the same process as used in Case 1 without decomposition. \medskip

\noindent\textbf{Case 3:} ${2/3}\leq p\leq \infty$ and $1\leq q\leq2.$

Since $1\leq q\leq2,$ by embedding theorem, we immediately have
\[F\dot{B}^{2-\frac{3}{p}}_{p,q}(\RR^3)\hookrightarrow L^2_{\rm uloc}(\RR^3).\]
By Remark \ref{rem.uloc} again and existence of singularities, we can get the required  result by Mimicking  the proof in Case 1 without decomposition.

 Theorem \ref{Thm} is thus proved.\qed



\section*{Acknowledgments}
  This work is supported in part by the National Natural Science Foundation of China
 under grant  No.11671047, and No.11501020.


\end{document}